\documentclass[11pt]{article}
\usepackage{amscd, amssymb, amsmath, amsthm, amsfonts}
\usepackage{latexsym, graphics, graphicx,psfrag}
\usepackage[all]{xy}

\makeatletter
\renewenvironment{proof}[1][\proofname]{\par
    \pushQED{\qed}%
    \normalfont \topsep6\p@\@plus6\p@ \labelsep1em\relax
    \trivlist
    \item[\hskip\labelsep\noindent
         #1.]\ignorespaces
}{%
    \popQED\endtrivlist\@endpefalse
}
\makeatother

\newtheoremstyle{mythm}{}{}
    {\rmfamily}{}{\bfseries}{ }{1ex}{}
\theoremstyle{mythm}
\newtheorem{thm}{Theorem}[section]
\newtheorem{col}[thm]{Corollary}
\newtheorem{lem}[thm]{Lemma}
\newtheorem{prop}[thm]{Proposition}

\newtheoremstyle{mydef}{}{}
    {\rmfamily}{\parindent}{\ttfamily}{ }{1ex}{}    %{\thmnote{#3}}
\theoremstyle{mydef}
\newtheorem{defn}[thm]{D{\footnotesize EFINITION}}

\newtheorem{exm}[thm]{E{\footnotesize XAMPLE}}

\numberwithin{equation}{section}

\begin{document}
\date{}
\title{Self-mapping degrees of torus bundles and torus semi-bundles
\footnotetext{\text{2000 Mathematics Subject Classification.}
Primary 57M10; Secondary 55M25.} }

\author{\bf Hongbin SUN, Shicheng WANG and Jianchun WU}

 \maketitle
%\address{School of Mathematics Science, Peking University\\Beijing 100871, People's Republic of China}
%\email{}%

%\thanks{}%
%\subjclass{}%
%\keywords{}%

%\dedicatory{}%
%\commby{}%
% ----------------------------------------------------------------
\maketitle
\begin{abstract} Each closed oriented $3$-manifold $M$ is
naturally associated with a set of integers $D(M)$, the degrees of
all self-maps on $M$. $D(M)$ is determined for each torus bundle and
torus semi-bundle $M$. The structure of torus semi-bundle is studied
in detail. The paper is a part of a project to determine $D(M)$ for
all 3-manifolds in Thurston's picture.
\\\

%\textbf{Subject class:} 57M50, 57M10, 55M20
\end{abstract}
% ---------------------------------------------------------------
\tableofcontents
\section{Introduction}

\subsection {Background.}

Each closed oriented $n$-manifold $M$ is naturally associated with a
set of integers, the degrees of all self-maps on $M$, denoted as
$D(M)=\{deg(f)\ |\  f :M\to M\}$.

Indeed the calculation of $D(M)$ is a classical topic appeared in
many literatures. The result is simple and well-known for dimension
$n=1,2$, and for dimension $n>3$, there are many interesting special
results (see \cite{DW} and references therein), but it is difficult
to get general results, since there are no classification results
for manifolds of dimension $n>3$.

The case of dimension 3 becomes attractive in the topic and it is
possible to calculate $D(M)$ for any closed oriented 3-manifold $M$.
Since Thurston's geometrization conjecture, which seems to be
confirmed, implies that closed oriented 3-manifolds can be
classified in reasonable sense.

Thurston's geometrization conjecture claims that the each
Jaco-Shalen-Johanson decomposition piece of a prime 3-manifold
supports one of the eight geometries which are $H^3$,
$\widetilde{PSL}(2,R)$, $H^2\times E^1$, Sol, Nil, $E^3$, $S^3$ and
$S^2\times E^1$ (for details see \cite{Th} and \cite{Sc}). Call a
closed orientable 3-manifold $M$ is {\it geometrizable} if each
prime factor of $M$ meets Thurston's geometrization conjecture.

A  known rather general fact about $D(M)$ for geometrizable
3-manifolds is the following:

\begin{thm}[\cite{Wa}, Corollary 4.3]\label{before}
Suppose $M$ is a geometrizable  3-manifold. Then $M$ admits a
self-map of degree larger than 1 if and only if $M$ is either

(1) covered by a torus bundle over the circle, or

(2) covered by an $F\times S^1$ for some compact surface $F$ with
$\chi(F)< 0$, or

(3) each prime factor of $M$ is covered by $S^3$ or $S^2\times E^1$.
\end{thm}

The proof of the "only if" part in Theorem \ref{before} is based on
the theory of simplicial volume, and various results on 3-manifold
topology and group theory. The proof of "if" part in Theorem
\ref{before} is a sequence of elementary constructions, which were
essentially known before.

Hence for any $M$ not listed in Theorem \ref{before},  $D(M)$ is
either $\{0,1,-1\}$ or $\{0,1\}$, which depends on whether $M$
admits a self map of degree $-1$ or not. To determine $D(M)$ for
geometrizable 3-manifolds listed in Theorem \ref{before}, let's have
a close look of those 3-manifolds from geometric and topological
aspects.

Among Thurston's eight geometries, six of them belong to the list in
Theorem \ref{before}.
 3-manifolds in (1) are exactly
those supporting either $E^3$, or Sol or Nil geometries.  $E^3$
3-manifolds, Sol 3-manifolds, and some Nil 3-manifolds are torus
bundles or semi-bundles; Nil 3-manifolds which are not torus bundles
or semi-bundles are Seifert spaces having Euclidean orbifolds with
three singular points. 3-manifolds in (2) are exactly those support
$H^2\times E^1$ geometry; 3-manifolds supporting $S^3$ or $S^2\times
E^1$ geometries form a proper subset of (3).

For 3-manifold $M$ with $S^3$-geometry, $D(M)$ has been presented
recently in \cite{Du} in term of the orders of $\pi_1(M)$ and its
elements (and determined earlier in \cite{HKWZ} when the maps induce
automorphisms on $\pi_1$). Note an algorithm is given to calculate
the degree set of maps between $S^3$-manifolds in term of their
Seifert invariants \cite{MP}.

To determine $D(M)$ for the remaining geometrizable 3-manifolds $M$,
the main task is to solve the question for the following three
groups ($D(M)$ is rather easy to determine for Seifert manifold $M$
supporting $H^2\times E^1$ or $S^2\times E^1$ geometry):

(a) torus bundles and semi-bundles;

(b) Nil Seifert manifolds not in (a);

(c) connected sums of 3-manifolds in (3) do not supporting $S^3$ or
$S^2\times E^1$ geometries.

Indeed $D(M)$ for $M$ in (a) will be determined in this paper
(hopefully all the remaining cases will be solved in a forthcoming
paper by the authors and Hao Zheng).

\subsection{Main result.} In this paper we calculate $D(M)$ for 3-manifold
$M$ which is either a torus bundle or semi-bundle. To do this, we
need first to coordinate torus bundles and semi-bundles by integer
matrices in Propositions \ref{coord-1} and \ref{coord-2}, then state
the results of $D(M)$ in term of those matrices in Theorems
\ref{main-1} and \ref{main-2}.

{\ttfamily{C{\footnotesize ONVENTION}:}} (1) To simplify notions,
for a diffeomorphism $\phi$ on torus $T$, we also use $\phi$ to
present its isotopy class and its induced 2 by 2 matrix on
$\pi_1(T)$ for a given basis.

(2) Each 3-manifold $M$ is oriented, and each 3-submanifold of $M$
and its boundary have induced orientations.

(3) Suppose $S$ (resp. $P$) is a properly embedded surface (resp. an
embedded 3-manifold) in a 3-manifold $M$.  We use $M\setminus S$
(resp. $M\setminus P$) to denote the resulting manifold obtained by
splitting $M$ along $S$ (resp. removing $\text{int} P$, the interior
of $P$).

\begin{defn} A {\it torus bundle} is $M_\phi = T \times I
/(x,1) \sim (\phi(x),0)$ where $\phi$ is a self-diffeomorphism of
the torus $T$ and $I$ is the interval $[0,1]$.

For a torus bundle $M_\phi$, we can isotopic $\phi$ to be a linear
diffeomorphism, which means $\phi \in GL_2 (\mathbb{Z})$ while not
changing $M_\phi$. Since we consider the orientation preserving case
only, $\phi$ must be in the special linear group $SL_2(\mathbb{Z})$.
\end{defn}

\begin{prop}\label{coord-1}
(1) $M_\phi$ admits $E^3$ geometry if and only if $\phi$ is
periodical, or equivalently $\phi$ is conjugate to one of the
following matrices $\left(
                                                  \begin{array}{cc}
                                                    1 & 0 \\
                                                    0 & 1 \\
                                                  \end{array}
                                                \right)$,
                                                $\left(
                                                  \begin{array}{cc}
                                                    -1 & 0 \\
                                                    0 & -1 \\
                                                  \end{array}
                                                \right)$,
                                                 $\left(
                                                  \begin{array}{cc}
                                                    -1 & -1 \\
                                                    1 & 0 \\
                                                  \end{array}
                                                \right)$,
                                                $\left(
                                                  \begin{array}{cc}
                                                    0 & -1 \\
                                                    1 & 0 \\
                                                  \end{array}
                                                \right)$,
                                                $\left(
                                                  \begin{array}{cc}
                                                    0 & -1 \\
                                                    1 & 1 \\
                                                  \end{array}
                                                \right)$ of finite
                                                 order
                                                1,2,3,4 and 6 respectively;

(2) $M_\phi$ admits Nil geometry if and only if $\phi$ is reducible,
 or equivalently $\phi$ is conjugate to $\pm\left(
                           \begin{array}{cc}
                             1 & n \\
                             0 & 1 \\
                           \end{array}
                         \right)$ where $n\neq 0$;

(3) $M_\phi$ admits Sol geometry if and only if $\phi$ is Anosov or
equivalently $\phi$ is conjugate to $\left(
                           \begin{array}{cc}
                             a & b \\
                             c & d \\
                           \end{array}
                         \right)$ where $|a+d|>2, ad-bc=1$.
\end{prop}
\begin{proof}
 see  \cite{Ha}.
 \end{proof}
\begin{defn} Let $K$ be  the Klein bottle and
$N=K\widetilde{\times}I$ be the twisted $I$-bundle over $K$. A {\it
torus semi-bundle} $N_\phi=N\bigcup_\phi N$ is obtained by gluing
two copies along their  torus boundary $\partial N$ via a
diffeomorphism $\phi$. Note $N_\phi$ is foliated by tori parallel to
$\partial N$ with a Klein bottle at the core of each copy of $N$.
\end{defn}

Let $(x,y,z)$ be the coordinate of $S^1\times S^1\times I$. Then
$N=S^1\times S^1\times I/\tau$, where $\tau$ is an orientation
preserving involution such that  $\tau(x,y,z)=(x+\pi,-y,1-z)$, and
we have the double covering $p: S^1\times S^1\times I\to N$.  Let
$C_x$ and $C_y$ be the two circles on $S^1\times S^1\times \{1\}$
defined by $y$ to be constant and $x$ to be constant, see Figure 1.
\begin{center}
\psfrag{x}[]{$x$} \psfrag{y}[]{$y$} \psfrag{z}[]{$z$}
\includegraphics[height=1.5in]{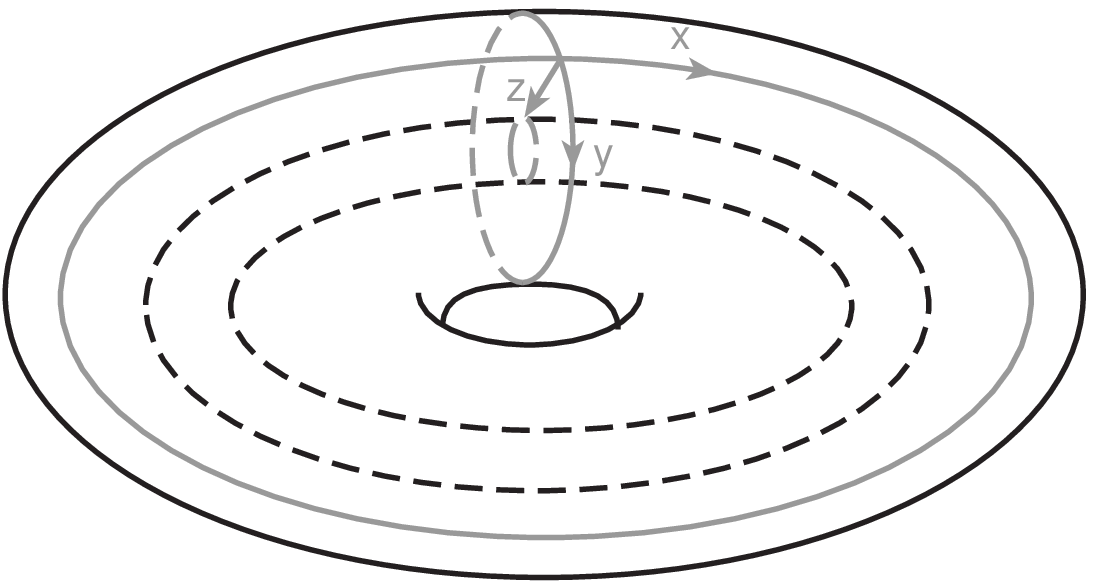}

\vskip 0.5 truecm \centerline{Figure 1: Coordinates of $S^1\times
S^1\times I$}
\end{center}
 Denote by $l_0=p(C_x)$ (0 slope) and
$l_\infty= p(C_y)$ ($\infty$ slope) on $\partial N$. {\it A
canonical coordinate} is an orientation of $l_0\cup l_\infty$, hence
there are four choices of canonical coordinate on $\partial N$. Once
canonical coordinates on each $\partial N$ are chosen, $\phi$  is
identified with an element $\left(\begin{array}{cc}
                                                    a & b \\
                                                    c & d \\
                                                  \end{array}
                                                \right)$ of $GL_2(\mathbb{Z})$ given by
$\phi$ $(l_0,l_\infty)$= $(l_0,l_\infty)$ $\left(\begin{array}{cc}
                                                    a & b \\
                                                    c & d \\
                                                  \end{array}
                                                \right)$.

\begin{prop}\label{coord-2}
With suitable choice of canonical coordinates of $\partial N$, we
have:

(1) $N_\phi$ admits $E^3$ geometry if and only if $\phi=\left(
                                                  \begin{array}{cc}
                                                    1 & 0 \\
                                                    0 & 1 \\
                                                  \end{array}
                                                \right)$ or
                                                 $\left(
                                                  \begin{array}{cc}
                                                    0 & 1 \\
                                                    1 & 0 \\
                                                  \end{array}
                                                \right)$;

(2) $N_\phi$ admits Nil geometry if and only if $\phi=\left(
\begin{array}{cc}
1 & 0 \\
z & 1 \\
\end{array}
\right)$, $\left(
\begin{array}{cc}
0 & 1 \\
1 & z \\
\end{array}
\right)$ or $\left(
\begin{array}{cc}
1 & z \\
0 & 1 \\
\end{array}
\right)$ where $z\neq 0$;

(3) $N_\phi$ admits Sol geometry if and only if $\phi=\left(
                           \begin{array}{cc}
                             a & b \\
                             c & d \\
                           \end{array}
                         \right)$ where $abcd\neq 0, ad-bc=1$.

Moreover a torus semi-bundle $N_\phi$ is also a torus bundle if and
only if
                                                $\phi=\left(
\begin{array}{cc}
1 & 0 \\
z & 1 \\
\end{array}
\right)$ under suitable choice of canonical coordinates

\end{prop}
We will prove Proposition \ref{coord-2} in Section 2.

\begin{thm}\label{main-1}
Using matrix coordinates given by Proposition \ref{coord-1},
$D(M_\phi)$ is listed in table 1 for torus bundle $M_\phi$, where
$\delta (3)=\delta (6)=1, \delta (4)=0$.
\begin{table}[!h]
\begin{center}
\begin{tabular}{|c|c|c|}
  \hline
  $M_\phi$ & $\phi$ & $D(M_\phi)$\\
  \hline
  $E^3$ & finite order $k=1,2$ & $\mathbb{Z}$\\
  \hline
  $E^3$  & finite order $k=3,4,6$ & $\{(kt+1)(p^2-\delta (k) pq+q^2)|\ t,p,q\in\mathbb{Z}\}$ \\
  \hline
  Nil  &  $\pm\left(\begin{array}{cc}1 & 0 \\n & 1 \\ \end{array}\right), n\neq 0$ & $\{l^2|\ l\in\mathbb{Z}\}$ \\
  \hline
  Sol &  $\left(\begin{array}{cc}a & b \\c & d \\ \end{array}\right), |a+d|>2$ & $\{p^2+\frac{(d-a)pr}{c}-\frac{br^2}{c}|\ p,r\in\mathbb{Z},$ \\ & & either $\frac{br}{c}, \frac{(d-a)r}{c} \in
\mathbb{Z}$ or
 $\frac{p(d-a)-br}{c} \in \mathbb{Z}\}$ \\
  \hline
\end{tabular}
\end{center}
\caption{degrees of self maps of orientable torus bundles}
\end{table}
\end{thm}
\begin{thm}\label{main-2}
Using matrix coordinates given by Proposition \ref{coord-2},
$D(N_\phi)$ is listed in table 2 for torus semi-bundle $N_\phi$,
where $\delta(a,d)= \frac{ad}{gcd(a,d)^2}$.
\begin{table}[!h]
\begin{center}
\begin{tabular}{|c|c|c|}
  \hline
  $N_\phi$ & $\phi$& $D(N_\phi)$\\
   \hline
  $E^3$ & $\left(\begin{array}{cc}1 & 0 \\0 & 1 \\ \end{array}\right)$ & $\mathbb{Z}$ \\
  \hline
  $E^3$ & $\left(\begin{array}{cc}0 & 1 \\1 & 0 \\ \end{array}\right)$ & $\{2l+1|\ l\in\mathbb{Z}\}$ \\
  \hline
  Nil & $\left(\begin{array}{cc}1 & 0 \\z & 1 \\ \end{array}\right), z\neq 0$ & $\{l^2|\ l\in\mathbb{Z}\}$ \\
  \hline
  Nil & $\left(\begin{array}{cc}0 & 1 \\1 & z \\
  \end{array}\right)$ or $\left(\begin{array}{cc}1 & z \\0 & 1 \\ \end{array}\right), z\neq 0$ & $\{(2l+1)^2|\ l\in\mathbb{Z}\}$ \\
  \hline
  Sol & $\left(\begin{array}{cc}a & b \\c & d \\ \end{array}\right), abcd\neq 0, ad-bc=1$&
  $\{(2l+1)^2|\ l\in\mathbb{Z}\}$, if $\delta(a,d)$ is even\ or\\ & &
       $\{(2l+1)^2|\ l\in\mathbb{Z}\}\bigcup\{(2l+1)^2\cdot \delta(a,d)$\\ & &
$|\ l\in\mathbb{Z}\}$, if $\delta(a,d)$ is odd \\
  \hline
\end{tabular}
\end{center}
\caption{degrees of self maps of torus semi-bundles}
\end{table}
\end{thm}

\subsection{Remark on orientation reversing homeomorphisms.}

Suppose $M$ is a torus bundle or semi-bundle. Then any non-zero
degree map is homotopic to a covering (\cite{Wa} Cor 0.4). Hence if
$-1\in D(M)$ (which is computable by Theorems \ref{main-1} and
\ref{main-2}), then $M$ admits an orientation reversing self
homeomorphism.

If $M$ is a torus semi-bundle, or $M$ supports the geometry of
either $E^3$ or Nil, then when $M$ admits an orientation reversing
self homeomorphism is explicitly presented in the following:

\begin{col}\label{or-r}
(1) A torus semi-bundle $N_\phi$ admits an orientation reversing
homeomorphism if and only if $\phi$ is either $\left(
\begin{array}{cc}
1 & 0 \\
0 & 1 \\
\end{array}
\right)$, or $\left(
\begin{array}{cc}
0 & 1 \\
1 &  0\\
\end{array}
\right)$, or $\left(
                           \begin{array}{cc}
                             a & b \\
                             c & -a \\
                           \end{array}
                         \right)$ where $abc\neq 0$.

(2) A torus bundle $M_\phi$ supporting $E^3$ geometry admits an
orientation reversing homeomorphism if and only if $\phi$ is either
$\left(
\begin{array}{cc}
1 & 0 \\
0 & 1 \\
\end{array}
\right)$, or $\left(
\begin{array}{cc}
-1 & 0 \\
0 &  -1\\
\end{array}
\right)$.

(3) If $M$ supports Nil geometry, then $M$ admits no orientation
reversing homeomorphism.
\end{col}

For torus bundle with given Anosov monodromy, even we can calculate
whether $-1 \in D(M_\phi)$, but there seems no explicit description
as in Corollary \ref{or-r}. (The referee informed us that there is a
 convenient description of when $-1\in D(M_\phi)$, see Lemma 1.7,
\cite{Sa})

\begin{exm} For the torus bundle $M_\phi$, $\phi=\left(
                           \begin{array}{cc}
                             2 & 1 \\
                             1 & 1 \\
                           \end{array}\right)$,  $-1 \in D(M_\phi)$.
Indeed for $\phi=\left(\begin{array}{cc}
                             a & b \\
                             c & d \\
                           \end{array}\right)$, if  $|a+d|=3$,
                           then $-1 \in
                           D(M_\phi)$.
Since $p^2+\frac{d-a}{b} pr-\frac{c}{b} r^2=-1$ has solution
$p=1-d$, $r=b$
             when $a+d=3$, and solution $p=-1-d$, $r=b$ when $a+d=-3$.
\end{exm}

\begin{exm} For the torus bundle $M_\phi$, $\phi=\left(
                           \begin{array}{cc}
                             2 & 3 \\
                             1 & 2 \\
                           \end{array}\right)$,  $-1 \notin D(M_\phi)$.
Indeed for $\phi=\left(\begin{array}{cc}
                             a & b \\
                             c & d \\
                           \end{array}\right)$, if  $a+d\pm 2$ has prime decomposition $p_1^{e_1}...p_n^{e_n}$ such that
                           $p_i=4l+3$ and $e_i=2m+1$ for some $i$,
                            then $-1 \notin
                           D(M_\phi)$. Since if the equation $p^2+\frac{d-a}{b} pr-\frac{c}{b}
                           r^2=-1$ has integer solution,
                           $\frac{((a+d)^2-4)r^2-4b^2}{b^2}$ should
                           be a square of rational number. That is $((a+d)^2-4)r^2-4b^2=s^2$ for some integer
                           $s$.
           Therefore $(a+d+2)(a+d-2)r^2$ is a sum of two squares. By
           a fact in elementary number theory,
            neither $a+d+2$ nor $a+d-2$ has $4k+3$ type prime factor with odd power (see  page 279, \cite{IR}).
\end{exm}

\begin{exm} Note if $-1 \in D(M)$, then $k\in D(M)$
implies $-k\in D(M)$. For the torus bundle $M_\phi$, $\phi=\left(
                           \begin{array}{cc}
                             2 & 1 \\
                             1 & 1 \\
                           \end{array}\right)$, among the first 20
                            integers $>0$, exactly $1, 4, 5, 9,
                           11 , 16, 19, 20\in D(M_\phi)$.
\end{exm}

\subsection {Organization of the paper.} Theorems \ref{main-1} and \ref{main-2}
will be proved in Sections 3 and 4 respectively. To prove these
theorems, we need have a careful look of the structures of torus
bundle and semi-bundles.  This is carried out in Section 2.

We explain more about Section 2. The most convenient and useful
reference for us is "Notes on basic 3-manifold topology" by Hatcher
\cite{Ha}, which is not formally published, but widely circulated
(see \newblock http://www.math.cornell.edu/$\sim$hatcher/.) In
particular Chapter 2 of \cite{Ha} is devoted to the study of torus
bundles and semi-bundles. Theorems \ref{thm21} and \ref{thm23} about
classifications of torus bundles and semi-bundles are quoted from
\cite{Ha} directly. It seems that the proof of Theorem \ref{thm23}
in \cite{Ha} missed an existed and rather complicated case, so we
rewrite a proof for it (most parts still follow that in \cite{Ha}).
Lemma \ref{lm5} studies incompressible surfaces in torus
semi-bundle, which relies on the proof of Theorem \ref{thm23}. Then
Proposition \ref{coord-2} is proved by using Theorem \ref{thm23},
Lemma \ref{lm5}, and Lemma \ref{lm2} which presents the relation
between gluing maps of a torus semi-bundles and its torus bundle
double covers. Finally, Theorem \ref{thm25} studies lifting of maps
between torus semi-bundles to their torus bundle double covers.

\section{Structures of orientable torus bundles and semi-bundles}

\subsection{Some elementary facts.}

All facts in this sub-section are known, and one can find them in
\cite{He}, or more directly in \cite{Ha}.

\begin{defn}
Suppose an oriented 3-manifold $M'$ is a circle bundle with a given
section $F$, where $F$ is a compact surface with boundary components
$c_1,...,c_n,...c_{n+b}$ with $n>0$. On each boundary component of
$M'$, orient $c_i$ and the circle fiber $l_i$ so that the product of
their orientations match with the induced orientation of $M'$. Now
attaching $n$ solid tori $S_i$ to the first $n$ boundary tori of
$M'$ so that the meridian of $S_i$ is identified with slope
$r_i=a_ic_i+ b_il_i$ with $a_i>0$. Denote the resulting manifold by
$M$ which has the Seifert fiber structure extended from the circle
bundle structure of $M'$.

We will denote this Seifert fibering of $M$ by $M(\pm
g,b;r_1,\cdots,r_s)$ where $g$ is the genus of the section $F$ of
$M$, with the sign $+$ if $F$ is orientable and $-$ if $F$ is
nonorientable,  here 'genus' of nonorientable surfaces means the
number of $RP^2$ connected summands. When $b=0$, call $e(M)=\sum_1^s
r_i$ the Euler number of the Seifert fiberation.
\end{defn}

\begin{center}
\psfrag{p}[]{$\rho$} \psfrag{o}[]{$l_0$} \psfrag{i}[]{$l_\infty$}
\psfrag{k}[]{$K$} \psfrag{t}[]{$T$} \psfrag{a}[]{$(a)$}
\psfrag{b}[]{$(b)$}
\includegraphics[height=1.5in]{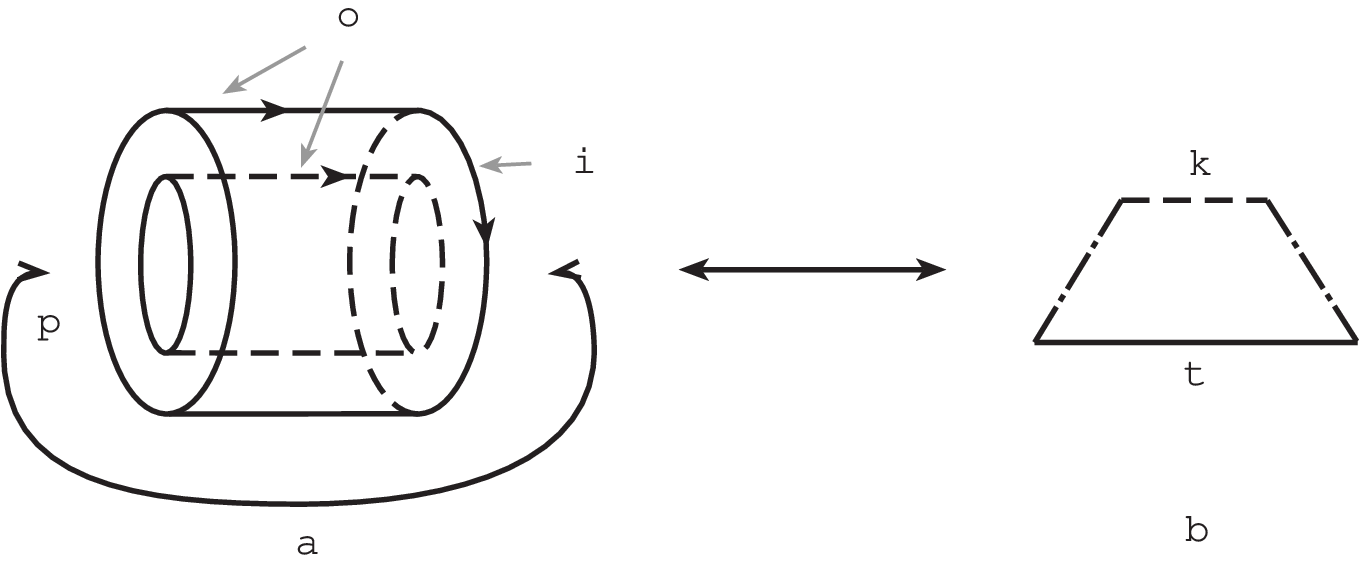}\\

\vskip 0.5 truecm
 \centerline{Figure 2: Coordinates of $\partial N$}
\end{center}

Another view of $N$ described in Figure 2(a): $N$ is obtained from
$S^1\times I \times I$ by identifying $S^1\times I \times \{0\}$
with $S^1\times I \times \{1\}$ via a diffeomorphism $\rho$ which
reflects both the $S^1$ and $I$ factors. Figure 2(b) is a schematic
picture of $N$ which will be used in the paper.

We list some properties of $N$ as:

\begin{lem} \label{basic}

(1) $N$ has two types of Seifert fiber structures:

I: $M(0,1;\frac{1}{2},-\frac{1}{2})$ in which  $l_0$  on $\partial
N$ is a regular fiber and $l_\infty$ is the boundary of the section
defining the Seifert invariant.

II: $M(-1,1;)$ in which  $l_\infty$ on $\partial N$ is a regular
fiber and $l_0$ is the boundary of the section defining the Seifert
invariant.

(2) $N$ has three types of essential (orientable, incompressible,
$\partial$-incompressible) surfaces:

I. a torus parallel to $\partial N$.

II. an annulus whose boundary is $l_\infty$  in $\partial N$ (Figure
3(a)) which does not separate  $N$.

III. an annulus whose boundary is $l_0$ in $\partial N$ (Figure
3(b)) which separates  $N$.

(3) Suppose $M$ is a torus bundle or semi-bundle and $F$ is a closed
incompressible surface in $M$, then $F$ is union of parallel tori.
\end{lem}

\begin{center}
\psfrag{p}[]{$\rho$} \psfrag{q}[]{$\rho$} \psfrag{a}[]{$(a)$}
\psfrag{b}[]{$(b)$}
\includegraphics[height=1.5in]{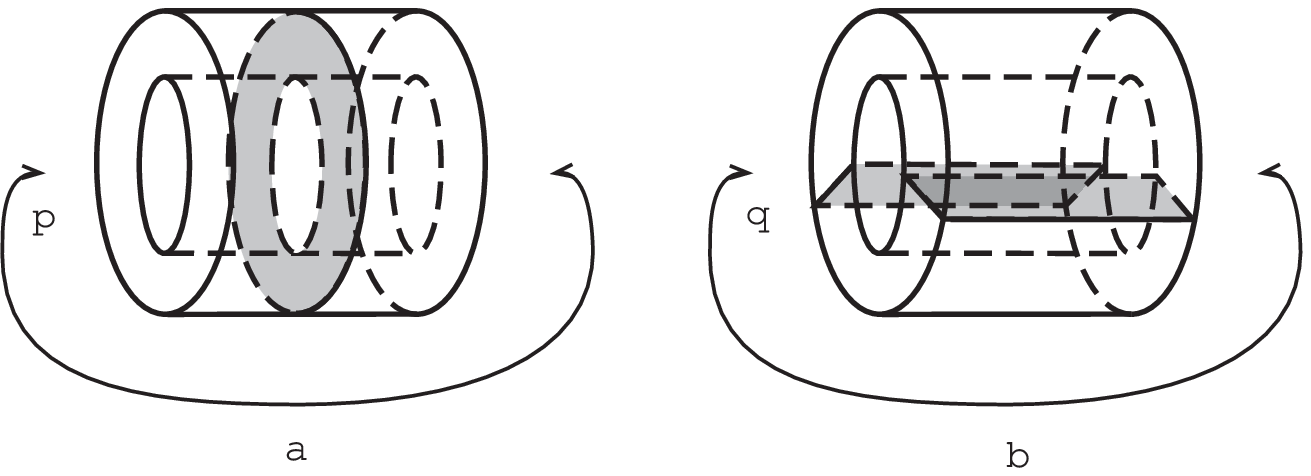}\\
\vskip 0.5 truecm \centerline{Figure 3: Essential surface in $N$}
\end{center}

\subsection {Classifications of torus bundles and semi-bundles.}

 Orientable torus bundles and semi-bundles are classified by two
 theorems
 below.
\begin{thm}[\cite{GS}; \cite{Ha}, Theorem 2.6] \label{thm21}
An orientable torus bundle $M_\phi$ is diffeomorphic to $M_\psi$ if
and only if $\phi$ conjugates to $\psi^{\pm 1}$ in
$GL_2(\mathbb{Z})$.
\end{thm}

\begin{thm}[\cite{Ha}, Theorem 2.8] \label{thm23}
The torus semi-bundle $N_\phi$ is diffeomorphic to $N_\psi$ if and
only if $\phi=\left(\begin{array}{cc} \pm1 & 0 \\0 & \pm1\\
\end{array}\right) \psi^{\pm 1}\left(\begin{array}{cc} \pm1 & 0 \\0 & \pm1\\
\end{array}\right)$ in $GL_2(\mathbb{Z})$, with
independent choices of signs understood.
\end{thm}

\begin{proof} (We start the proof as that in \cite{Ha}.) Suppose $f:N_\phi \to N_\psi$ is a diffeomorphism and $T,T'$
are the torus fibers of $N_\phi,N_\psi$ respectively.
$N_\psi\setminus T'=N_1\bigcup N_2$ where $N_1$ , $N_2$ are
homeomorphic to $N$.

Since $f$ is a diffeomorphism, two components of $N_\psi\setminus
f(T)$ are both homeomorphic to $N$. We can isotope $f$, such that
every component of $f(T)\bigcap N_i$ is an essential surface in
$N_i,\ i=1,2$. So $f(T)\bigcap N_i$ is in the three types listed in
Lemma \ref{basic} (2). Thus either $f(T)$ is parallel to $T'$, or
$\psi$ takes $l_0$ or $l_\infty$ to $l_0$ or $l_\infty$.

Suppose $f(T)$ is parallel to $T'$. We can assume $f(T)=T'$. Then
$\phi$ must be obtained from $\psi$ by composing on the left and
right homeomorphisms of $\partial N$ which extend to homeomorphisms
of $N$. Such homeomorphisms must preserve both $l_0$ and $l_\infty$
(may reverse the directions), since $l_0$ is the unique slopes of
the boundaries of essential separating annulus and $l_\infty$ is the
unique slopes of the boundaries of essential non-separating annulus
in $N$. Theorem \ref{thm23} is proved in this situation.

Suppose $\psi$ takes $l_0$ or $l_\infty$ to $l_0$ or $l_\infty$.
Then there are three cases as below:

Case (1)  $\psi$ takes $l_\infty$ to $l_0$ (if $\psi$ takes $l_0$ to
$l_\infty$, then we consider $\psi^{-1}$),

Case (2) $\psi$ takes $l_\infty$ to $l_\infty$,

Case (3) $\psi$ takes $l_0$ to $l_0$.

(The proof in \cite{Ha} claims that only Case (3) is  possible,
while we show below that only Case (2) is impossible).

Case (1). Now $\psi$=$\left(
            \begin{array}{cc}
              z & 1 \\
              1 & 0 \\
            \end{array}
           \right)$, and $N_\psi=M(-1,0;1/2,-1/2,z)$, and $e(M)=z$. Note:

(i) $f(T)\bigcap N_1$ are $n$ parallel annuli $A_1,\dots,A_n$ of
type (II) (see Figure 4), which are located in a cyclic order in
$N$. Set $\partial A_i=a_i\bigcup a_i'$, then $2n$ circles
$a_1,\dots,a_n,a_1',\dots,a_n'$ are located in cyclic order in
$\partial N_1$.
\begin{center}
\psfrag{p}[]{$\rho$} \psfrag{a}[]{$a_i$} \psfrag{b}[]{$a_{i+1}$}
\psfrag{c}[]{$A_i$} \psfrag{d}[]{$A_{i+1}$}
\psfrag{e}[]{$a_{i}^{'}$} \psfrag{f}[]{$a_{i+1}^{'}$}
\includegraphics[height=1.5in]{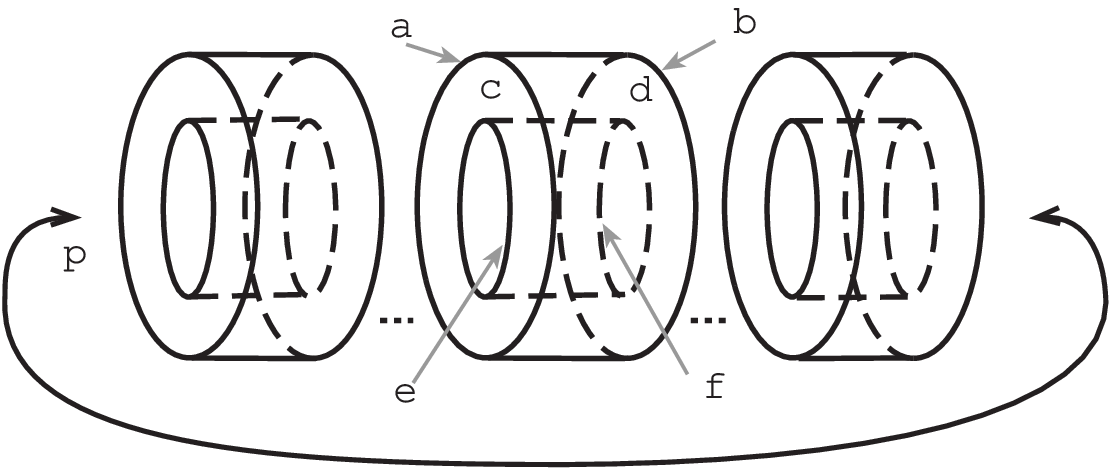}\\
\vskip 0.5 truecm
 \centerline{Figure 4: Cut $N_i$ through type (II) surfaces}
\end{center}

(ii) $f(T)\bigcap N_2$ are annuli $B_1,\dots,B_n$ of type (III)
 (see Figure 5), where $B_{i+1}$ is next
 to $B_i$, $i=1, ..., n-1$ in $N_2$. Set $\partial B_i=b_i\bigcup b_i'$
then $2n$ circles $b_1,\dots,b_n,b_n',\dots,b_1'$ are located in
cyclic order in $\partial N_2$.
\begin{center}
\psfrag{q}[]{$\rho$} \psfrag{a}[]{$b_1^{'}$} \psfrag{b}[]{$b_1$}
\psfrag{c}[]{$b_i^{'}$} \psfrag{d}[]{$b_i$} \psfrag{e}[]{$b_n^{'}$}
\psfrag{f}[]{$b_n$} \psfrag{k}[]{$B_1$} \psfrag{l}[]{$B_i$}
\psfrag{m}[]{$B_n$}
\includegraphics[height=2.5in]{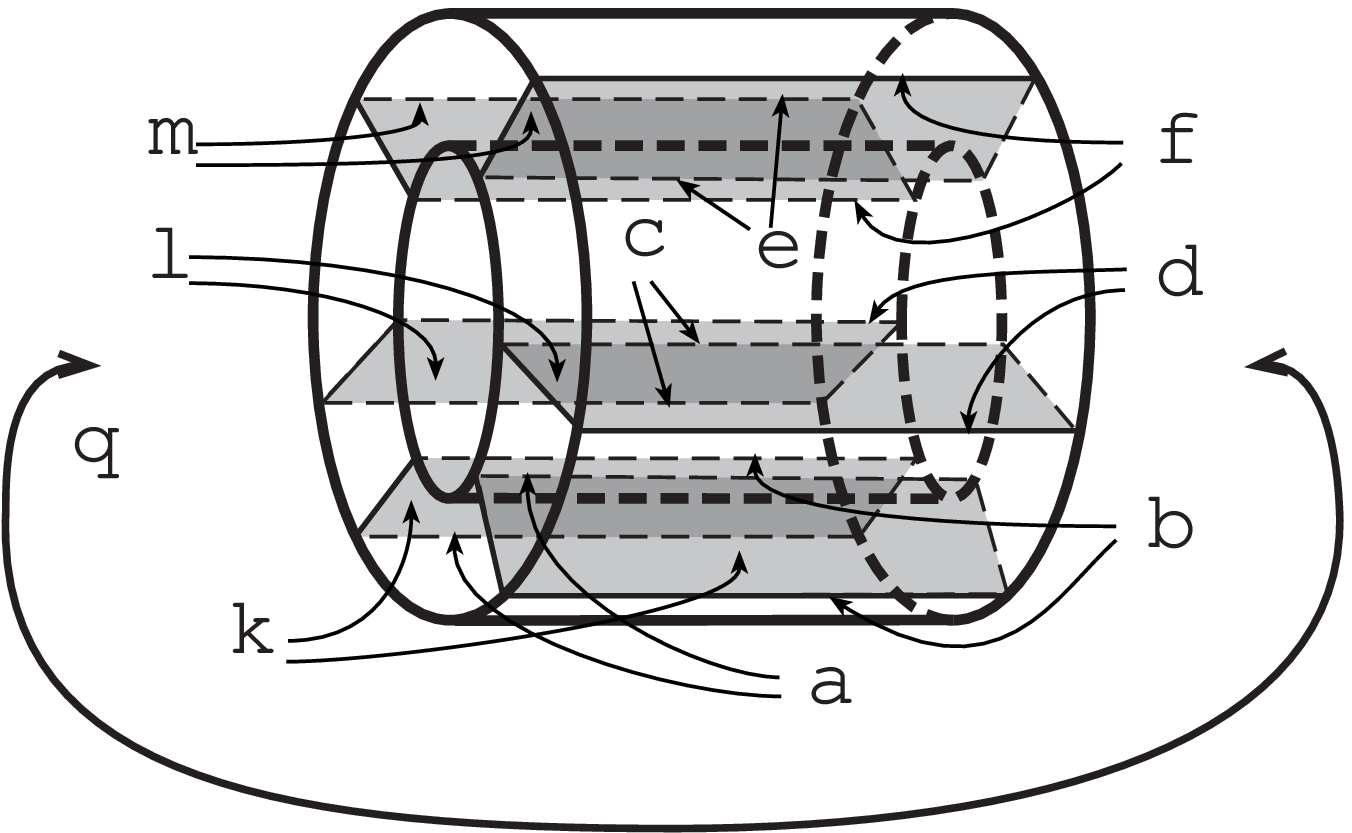}\\
\vskip 0.5 truecm \centerline{Figure 5: Cut $N_i$ through type (III)
surfaces}
\end{center}

If $n=1$, we can check that $\psi$ pastes $A_1$ and $B_1$ to a Klein
bottle, which contradicts the fact that $f(T)$ is torus. When $n>1$,
we can assume $\psi$ pastes $a_1$ to $b_1$ and pastes $a_2$ to
$b_2$, after reindexing $A_i$ if necessary. By the orders of
sequences of $a_1,\dots,a_n,a_1',\dots,a_n'$ and
$b_1,\dots,b_n,b_n',\dots,b_1'$ on $\partial N_1$ and $\partial
N_2$, we have $a_i$ is pasted to $b_i$, and $a_i'$ pasted to
$b_{n-i}'$, $i=1,\dots,n$. So $A_i$, $A_{n-i}$, $B_i$, $B_{n-i}$ are
pasted to one component of $f(T)$ in $N_\psi$, and $f(T)$ has
$[\frac{n+1}{2}]$ components. Since $f(T)$ is connected, we have
$n=2$.

Now $N_1\setminus f(T)$ can be presented as two I-bundles over
annulus: $I\times A_1$ and $I\times A_2$, where $f(T)\bigcap
N_1=A_1\cup A_2$, as in Figure 4. $N_2\setminus f(T)$ can be
presented as an I-bundle over annulus $I\times B$ as in Figure 6(a)
and two solid tori $P_1$ and $P_2$ with the core of
 $P_i\bigcap \partial N_2$ to be the $(2,1)$ curve of $\partial
 P_i$ as in Figure 6(b).

\begin{center}
\psfrag{p}[]{$\rho$} \psfrag{q}[]{$\rho$} \psfrag{a}[]{$(a)$}
\psfrag{b}[]{$(b)$}
\includegraphics[height=2in]{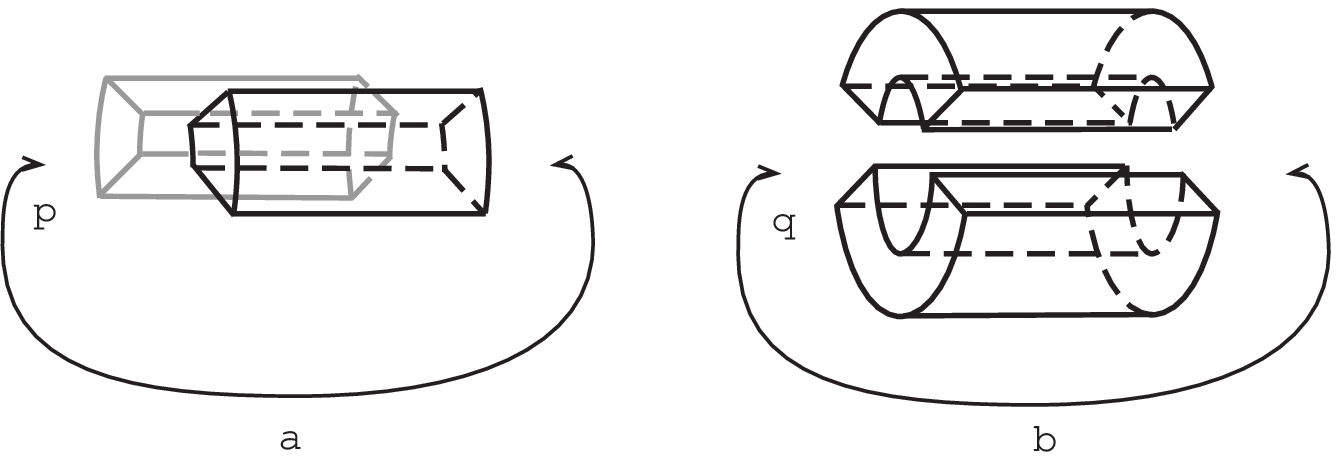}\\
\vskip 0.5 truecm \centerline{Figure 6}
\end{center}

If we glue those five pieces along $\partial N$, we get
 two components of $N_\psi\setminus f(T)$ which are $N_1'=P_1\cup_{\partial
N}I\times A_1\cup_{\partial N}P_2$ and $N_2'=I\times
A_2\cup_{\partial N} I\times B$ (re-index $A_i$ if needed), each of
them is a copy of $N$. Moreover under the inherited Seifert
structure of $N_\psi$, $N'_1=M(0,1;\frac{1}{2},-\frac{1}{2})$ and
$N'_2=M(-1,1;)$.

If we consider that $M(-1,0;1/2,-1/2,z)$ is obtained by identifying
$N_1'$ and $N_2'$ along $f(T)$, we get a new semi-bundle structure
so that $f(T)$ become a fiber torus. Since the Euler number of the
Seifert structure is $z$, the new gluing map must be $\left(
            \begin{array}{cc}
              z & 1 \\
              1 & 0 \\
            \end{array}
           \right)^{\pm 1}$. This reduces us to the situation that $f(T)$ is parallel to $T'$.

Case (2). Both $f(T)\bigcap N_i$  are type (II) surfaces, for
$i=1,2$ (Figure 4). Hence $f(T)\bigcap N_1$ is exactly as that in
Case (1) (i). Similarly, $f(T)\bigcap N_2$ are $n$ parallel annulus
$B_1,\dots,B_n$  located in a cyclic order in $N$. Set $\partial
B_i=b_i\bigcup b_i'$, then $2n$ circles
$b_1,\dots,b_n,b_1',\dots,b_n'$ are located in cyclic order in
$\partial N_2$.

We can assume $\psi$ paste $a_1$ to $b_1$ and paste $a_2$ to $b_2$
(re-index $\{B_i\}$ if needed). Then we have $a_i$ is pasted to
$b_i$, and $a_i'$ pasted to $b_i'$, $i=1,\dots,n$. So $A_i$ and
$B_i$ are pasted to one component of $f(T)$ in $N_\psi$. Since
$f(T)$ is connected, $n=1$. But here $f(T)$ does not separate
$N_\psi$, it is impossible.

Case (3). (We copy the proof of \cite{Ha} for this case.) Now
$\psi$=$\left(
            \begin{array}{cc}
              1 & z \\
              0 & 1 \\
            \end{array}
           \right)$, and $N_\psi=M(0,0;1/2,1/2,-1/2,-1/2,z)$, $e(N_\psi)=z$.  (Both $f(T)\bigcap N_i$ are type (III).)

We may assume that $f(T)$ has been isotoped to be either vertical or
horizontal in this Seifert fibering. Since a connected horizontal
essential surface is not separating, $f(T)$ must be vertical. Then
$f(T)$ must separate $M(0,0;1/2,1/2,-1/2,-1/2,z)$ into two copies of
$N$ both having the inherited  Seifert structure
$M(0,1;\frac{1}{2},-\frac{1}{2})$. We can rechoose the semi-bundle
structure so that $f(T)$ become a fiber torus. Then for the new
torus semi-bundle structure the gluing map must also be $\left(
            \begin{array}{cc}
              1 & z \\
              0 & 1 \\
            \end{array}
           \right)$. This reduces us to the situation that $f(T)$ is parallel to
$T'$.\end{proof}

\subsection {Incompressible surfaces.}

\begin{lem}[\cite{Ha}, Lemma 2.7]\label{lm3}For a torus bundle $M_\phi$, if $\phi$ is not conjugate to $\pm\left(
                                                                            \begin{array}{cc}
                                                                              1 & 0 \\
                                                                              n & 1 \\
                                                                            \end{array}
                                                                          \right)$,
 then any essential closed surface in $M_\phi$ is isotopic to a union of torus fibers.
\end{lem}

\begin{lem}\label{lm5}
If a torus semi-bundle $N_\phi$ has no torus bundle structure, then
any essential closed surface in $N_\phi$ is isotopic to copies of
torus fibers of a torus semi-bundle structure on $N_\phi$, which is
isomorphic to $N_\phi$.
\end{lem}
\begin{proof}  Let $F$ be an essential close surface in $N_\phi=N_1\bigcup N_2$.
By Lemma \ref{basic} (3), $F$ is a union of parllel tori. For our
purpose we may assume that $F$ is a torus. Isotope $F$ so that
$F\bigcap N_i$ is essential in $N_i$. Then each component of
$F\bigcap N_i$ must be in one of the three types listed in Lemma
\ref{basic}.

If  $F\bigcap N_i$ is of type (I), then the proof is finished.

 There are two cases remaining:

(a) Both $F\bigcap N_i$ are of type (II) for $i=1,2$ (Figure 4).
Then $N_i\setminus F$  are I-bundles over $N_i\cap F$. Gluing those
two $I$-bundles along $\partial N$ will get an I-bundle over $F$ and
$N_\phi$ is obtained from this I-bundle by identifying its top and
bottom, which provides a torus bundle structure of $N_\phi$.

(b) Some $F\bigcap N_i$ is of type (III), say $i=2$ (Figure 5). Then
$F$ is the same as $f(T)$  either in Case (1) or Case (3) of the
proof of Theorem \ref{thm23}, depends on $F\bigcap N_1$ is of type
(III) or type (II).

As indicated in the proof of Theorem \ref{thm23}, we can rechoose
the new torus semi-bundle structure $N_\psi$ so that $F$ become a
fiber torus; moreover if choosing suitable coordinates, we can make
$\psi$ to be $\phi$.
\end{proof}

\subsection{Coordinates of torus semi-bundles.}

Call a map $g:(M,\partial M)\rightarrow (M',\partial M')$ is {\it
proper} if $g^{-1}(\partial M') \subset \partial M$.

\begin{lem} If $V=T\times I$ with the two boundaries $T^+,T^-$ and
$g: (V,T^+,T^-) \rightarrow (N,\partial N)$ is a proper map, then
$(g|_{T^+})_*=\tau_*\cdot (g|_{T^-})_*$, where
$\tau_*=\left(\begin{array}{cc}1 & 0 \\0 & -1 \\
\end{array}\right)$.
\end{lem}
\begin{proof} Let $p:T\times I\to N$ be the double covering  and $\tau$ be the deck
transformation map.

Since $g_{*}(\pi_1(V))=(g|_{T^+})_*(\pi_1(T^+))\subset
\pi_1(\partial N)\subset \pi_1(N)$, thus $g$ can be lifted to a map
$\widetilde{g}: V\rightarrow T\times I$.
$$
 \xymatrix{
                &         (T\times I,T\times \{0\},T\times \{1\}) \ar[d]^{p} \ar[r]^{\tau} & (T\times I,T\times \{1\},T\times \{0\}) \ar[dl]^{p}    \\
 (V,T^+,T^-) \ar[ur]^{\widetilde{g}} \ar[r]_{g} & (N,\partial N)             }
$$

From the commuted diagram above, we have:
$$
\left\{ \begin{array}{l}
         g|_{T^-}=p|_{T\times \{1\}}\circ \widetilde{g}|_{T^-},\\
         g|_{T^+}=p|_{T\times \{1\}}\circ \tau|_{T\times \{0\}} \circ
         \widetilde{g}|_{T^+}.
         \end{array} \right.
$$
We can choose coordinate on $(T\times I,T\times \{0\},T\times
\{1\})$, such that $p|_{T\times \{1\}}=id$.

When considering fundamental group, we have
$(\widetilde{g}|_{T^-})_*=(\widetilde{g}|_{T^+})_*$. Thus by the
above equation:
$$
(g|_{T^+})_*=\tau_*\cdot (g|_{T^-})_*
$$
where $\tau_*=(\tau|_{T\times \{0\}})_*=\left(
                                        \begin{array}{cc}
                                          1 & 0 \\
                                          0 & -1 \\
                                        \end{array}
                                      \right)
$.
\end{proof}

\begin{lem}\label{lm2}
A torus semi-bundle $N_\phi$ is doubly covered by a torus bundle
$M_{\tau\phi\tau\phi^{-1}}$, where $\tau (x,y)=(x+\pi,-y)$ with
suitable choice of coordinate $(x,y)$ on the torus.
\end{lem}
\begin{proof}
Let $N_{\phi}=N_1\bigcup_{\phi}N_2$ with $\partial N_1 =\partial
N_2=T$. Let $p:M\rightarrow N_{\phi}$ be the double cover, where $M$
is a torus bundle, $p^{-1}(N_i)=M_i$ is homeomorphic to $T\times I$,
$p^{-1}(T)=T_1 \bigcup T_2$. Cut $M$ along $T_1,T_2$, get $M
\setminus T_1\bigcup T_2$. The two boundaries of $M_i$ are denoted
by $T_i$ and $T_i'$, $T_1$ is pasted to $T_2$ by $\psi$, $T_1'$ is
pasted to $T_2'$ by $\psi'$. Let $p_i=p|_{M_i}$. All of these are
shown in figure 7.

\begin{center}
\psfrag{a}[]{$T_1^{'}$} \psfrag{b}[]{$T_2^{'}$} \psfrag{c}[]{$T_1$}
\psfrag{d}[]{$T_2$} \psfrag{e}[]{$\psi^{'}$} \psfrag{f}[]{$\psi$}
\psfrag{g}[]{$\phi$} \psfrag{i}[]{$p_1$} \psfrag{j}[]{$p_2$}
\psfrag{k}[]{$M_1$} \psfrag{l}[]{$M_2$} \psfrag{m}[]{$N_1$}
\psfrag{n}[]{$N_2$}
\includegraphics[height=3in]{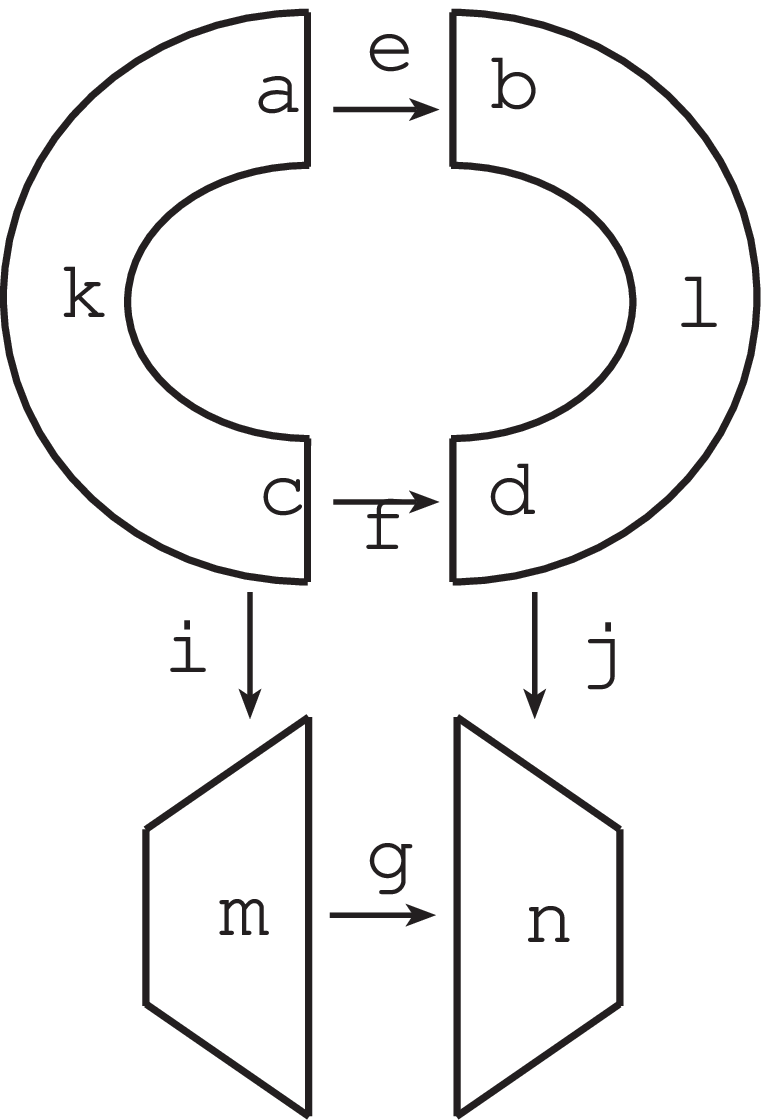}\\
\vskip 0.5 truecm \centerline{Figure 7: $N_\phi$ is double covered
by $M_{\tau\phi\tau\phi^{-1}}$}
\end{center}

We can choose coordinate on $T_1,T_2$, such that
$(p_i|_{T_i})_*=id$. Since $T_i'$ is parallel to $T_i$, we can
identify $\pi_1(T_i')$ with $\pi_1(T_i)$. By lemma 2.7, we have
$(p_i|_{T_i'})_*=\tau_*\cdot (p_i|_{T_i})_*$.

From  Figure 7, we know that

$$
\left\{ \begin{array}{l}
        (p_2|_{T_2})_*\circ \psi = \phi \circ (p_1|_{T_1})_*,\\
        (p_2|_{T_2'})_*\circ \psi' = \phi \circ (p_1|_{T_1'})_*.
\end{array} \right.
$$

Then we get

$$
\left\{ \begin{array}{l}
        \psi = \phi, \\
        \psi'=\tau \circ \phi \circ \tau.
\end{array} \right.
$$

Thus $M$ has the torus bundle structure
$M_{\psi'\psi^{-1}}=M_{\tau\phi\tau\phi^{-1}}.$
\end{proof}
By Theorem \ref{thm23}, and the fact that $\left(
            \begin{array}{cc}
              0 & 1 \\
              1 & z \\
            \end{array}
          \right)^{-1}$=$\left(
            \begin{array}{cc}
              -z & 1 \\
              1 & 0 \\
            \end{array}
          \right)$, with suitable choice of canonical coordinates of
$\partial N$, we can set $\phi$ is one of the four matrices below:

$\left(
            \begin{array}{cc}
              0 & 1 \\
              1 & z \\
            \end{array}
          \right)$,
          $\left(
            \begin{array}{cc}
              1 & z \\
              0 & 1 \\
            \end{array}
          \right)$, $\left(
            \begin{array}{cc}
              1 & 0 \\
              z & 1 \\
            \end{array}
          \right)$ and
          $\left(
                           \begin{array}{cc}
                             a & b \\
                             c & d \\
                           \end{array}
                         \right)$ where $abcd\neq 0, ad-bc=1$.

When $\phi$ is in the first three matrices, $N_\phi$ is a Seifert
manifold with Euler number $z$. $N_\phi$ is $E^3$ manifold if $z=0$
and is Nil manifold if $z\ne 0$. Now suppose $\phi$=$\left(
                           \begin{array}{cc}
                             a & b \\
                             c & d \\
                           \end{array}
                         \right)$ where $abcd\neq 0, ad-bc=1$. Then by Lemma \ref{lm2}, $N_\phi$ is double covered by
$M_{\tau\phi\tau\phi^{-1}}$. Since

$$(\tau \phi \tau \phi^{-1})_*=\tau_* \cdot \phi_* \cdot \tau_* \cdot
\phi_*^{-1}=\left(
              \begin{array}{cc}
                ad+bc & -2ab \\
                -2cd & ad+bc \\
              \end{array}
            \right),$$
we have

$$|Trace((\tau \phi \tau
\phi^{-1})_*)|=2|ad+bc|=2|ad-bc+2bc|=2|2bc+1|>2.$$

            By Proposition \ref{coord-1}, $M_{\tau\phi\tau\phi^{-1}}$ admits Sol
            geometry,
thus $N_\phi$ admits Sol geometry. The first part of Proposition
\ref{coord-2} is proved.

If $N_\phi$ also has torus bundle structure, it must have
non-separating essential torus. Recall the proof of Lemma \ref{lm5},
an essential torus in $N_\phi$ can be non-separating only if case
(a) is happened, and in this case $\phi=\left(
\begin{array}{cc}
1 & 0 \\
z & 1 \\
\end{array}
\right)$ under suitable choice of canonical coordinates, and
$N_\phi$ does have torus bundle structure. This finishes the
"moreover" part of Proposition \ref{coord-2}.

\subsection{Lifting automorphism from semi-bundle to bundle.}

\begin{thm}\label{thm25}
Suppose  $f: N_\phi \to N_\psi$ is a non-zero degree map and
$f^{-1}(T')$ is a union of copies of $T$, where $T,T'$ are the torus
fiber of $N_\phi,N_\psi$ respectively. Then we have commute diagram
\[ \begin{CD}
M@> \tilde{f} >> M'\\
@V p VV @VV p' V\\
N_\phi @> f >> N_\psi
\end{CD}  \]
where $M,M'$ are the torus bundle which are double covers of
$N_\phi,N_\psi$ respectively and $\tilde{f}:M \to M'$ is a lift of
$f$:
\end{thm}

\begin{proof} We only have to check $f_*(p_*(\pi_1(M))) \subset
p'_*(\pi_1(M'))$.

Let $\tilde{T},\tilde{T'}$ be one of the lifting of $T,T'$ in $M,M'$
respectively. In torus bundle $M$, we have the exact sequence:$$1
\to \pi_1(\tilde{T}) \to \pi_1(M) \to \pi_1(S^1) \to 1.$$ In torus
semi- bundle $N_\phi$, we have another exact sequence: $$1 \to
\pi_1(T) \to \pi_1(N_\phi) \to \mathbb{Z}_2 *\mathbb{Z}_2 \to 1.$$

Since $f^{-1}(T')$ is a union of copies of $T$, we can assume
$f(T)=T'$. Then we have the commuted diagram (every row is exact):
\[ \begin{CD}
1 @> >> \pi_1(\tilde{T}) @> \tilde{i_1} >> \pi_1(M) @> \tilde{j_1}
>>\pi_1(S^1)@> >>1\\
@. @V (p|)_*VV  @V p_*VV @V \bar{p}_*VV @.\\
1 @> >> \pi_1(T) @> i_1 >> \pi_1(N_\phi) @> j_1
>>\mathbb{Z}_2 *\mathbb{Z}_2@> >>1\\
@. @V (f|)_*VV  @V f_*VV @V \bar{f}_*VV @.\\
1 @> >> \pi_1(T') @> i_2 >> \pi_1(N_\psi) @> j_2
>>\mathbb{Z}_2 *\mathbb{Z}_2@> >>1\\
@.  @A (p'|)_*AA  @A p'_*AA @A \bar{p'}_*AA @.\\
1 @> >> \pi_1(\tilde{T'}) @> \tilde{i_2} >> \pi_1(M') @> \tilde{j_2}
>>\pi_1(S^1)@> >>1,\\
\end{CD}  \]
here $\bar{p}_*,\bar{p'}_*,\bar{f}_*$ are the maps among the
fundamental groups of the base spaces of fiber bundles induced by
the maps among the fundamental groups of the total spaces.

We present the group $\mathbb{Z}_2 *\mathbb{Z}_2$ by $<a,b\ |\
a^2=b^2=1>$ and choose the generator $a,b$ such that
$\bar{p}_*(1)=ab,\bar{p'}_*(1)=ab$ (here $1$ is the generator of
$\pi_1(S^1)$).

 Since $a^2=b^2=1$, so
$\bar{f}_*(a)^2=\bar{f}_*(b)^2=1$, then $\bar{f}_*(a),\bar{f}_*(b)$
must be of the form $ab\cdots ba$ or $ba\cdots ab$, and
$\bar{f}_*(ab)=(ab)^k$ or $(ba)^k=(ab)^{-k}$. So
$\bar{f}_*(\bar{p}_*(\pi_1(S^1))) \subset \bar{p'}_*(\pi_1(S^1))$.

For any $\alpha \in \pi_1(M)$, let $\beta=f_*(p_*(\alpha))$. Since
$j_2(\beta)=\bar{f}_*(\bar{p}_*(\tilde{j}_1(\alpha))) \in
\bar{p'}_*(\pi_1(S^1))$, and  there is $\gamma \in \pi_1(M')$ such
that $\bar{p'}_*(\tilde{j}_2(\gamma))=j_2(\beta)$, so
$$j_2(p'_*(\gamma)\cdot\beta^{-1})=\bar{p'}_*(\tilde{j}_2(\gamma))\cdot j_2(\beta^{-1})=j_2(\beta)\cdot j_2(\beta^{-1})=1.$$

Since $(p'|)_*$ is an isomorphism, there is $\delta \in
\pi_1(\tilde{T'})$ such that
$i_2((p'|)_*(\delta))=p'_*(\gamma)\cdot\beta^{-1}$. We have
$$p'_*(\tilde{i_2}(\delta^{-1})\cdot\gamma)=i_2((p'|)_*(\delta^{-1}))
\cdot p'_*(\gamma)=(p'_*(\gamma)\cdot\beta^{-1})^{-1}\cdot
p'_*(\gamma)=\beta.$$

So $f_*(p_*(\pi_1(M))) \subset p'_*(\pi_1(M'))$, thus $\tilde{f}$
exists.
\end{proof}

\section{The degrees of self maps of torus bundles}

We are going to prove Theorem \ref{main-1} (ref. Proposition
\ref{coord-1}). There are two cases to consider:
\\
\textbf{Case 1:} $\phi$ is conjugated to
$\pm\left(\begin{array}{cc}1 & 0
\\n & 1
\\\end{array}\right)$. Now $M_\phi$ is a seifert manifold whose Euler number of seifert fibering
$e(M_\phi)$ is equal to $n$.

\textbf{(1.I)} \ If $n=0$, $M_\phi$ is $T^3$ or
$S^1\widetilde{\times}S^1\widetilde{\times}S^1$. Here $\phi=$
$\pm\left(\begin{array}{cc}1 & 0
\\0 & 1
\\\end{array}\right)$, any $2 \times 2$ integer matrix $A$
commutes with $\phi$, so $M_\phi$ admits self maps of any degrees.

\textbf{(1.II)} If $n \neq 0$, for a none zero degree map $f: M_\phi
\rightarrow M_\phi$, by \cite[Corollary 0.4]{Wa}, $f$ is homotopic
to a covering map $g: M_\phi \rightarrow M_\phi$. We can choose a
suitable seifert fibering of $M_\phi$ such that $g$ is a fiber
preserving map. Denote the orbifold of $M_\phi$ by $O(M_\phi)$. By
\cite[Lemma 3.5]{Sc}, we have:
\begin{equation}\label{eq1}
\left\{ \begin{array}{l}
         e(M_\phi)=e(M_\phi)\cdot \frac{l}{m}, \\
         deg(g)=l\cdot m,
         \end{array} \right.
\end{equation}
where $l$ is the covering degree of $O(M_\phi)\rightarrow O(M_\phi)$
and $m$ is the fiber degree.

Since $e(M_\phi)\neq 0$, from equation (\ref{eq1}) we get $l=m$.
Thus $deg(f)=deg(g)$ is a square number. Conversely, given a square
number $l^2$, it is easy to construct  a covering map $f: M_\phi
\rightarrow M_\phi$ of degree $l^2$.
\\
\textbf{Case 2:} $\phi=\left(\begin{array}{cc}a & b
\\c & d\\\end{array}\right)$ is not conjugated to
$\pm\left(\begin{array}{cc}1 & 0 \\n & 1
\\\end{array}\right)$.

\begin{thm}\label{lm4} Suppose $\phi$ is not conjugated to
$\pm\left(\begin{array}{cc}1 & 0 \\n & 1
\\\end{array}\right)$
$M_\phi$ admits a self map of degree $l\neq 0$ if and only if there
exist a $2\times 2$ nondegenerate integer matrix $A$ and a positive
integer $k$ such that $l=k \cdot \epsilon \cdot det(A)$ and
$A\cdot\phi_*=(\phi^{\epsilon})_*^k\cdot A$ where $\epsilon=\pm 1$.
\end{thm}

\begin{proof}
For a torus fiber $T \in M_\phi$, $T$ is incompressible. Suppose $f:
M_\phi \rightarrow M_\phi$ is a self-map of degree $l\neq 0$. By
\cite[Lemma 6.5]{He}, $f$ is homotopic to $g: M_\phi \rightarrow
M_\phi$ such that $g^{-1}(T)$ is an incompressible surface of
$M_\phi$. Thus by Lemma \ref{lm3}, $g^{-1}(T)$ is isotopic to a
union of torus fibers.

Suppose $M_\phi\setminus g^{-1}(T)$ has k components $V_1,...,V_k$.
Each $V_i$ is a $T\times I$. Denote two torus boundary components of
$V_i$ by $T_i^+$ and $T_i^-$,  and the homeomorphism gluing $T_i^-$
to $T_{i+1}^+$ by $\psi_{i}$ see Figure 8. Then
$M_{\psi_k\circ...\circ\psi_1}=M_\phi$. By choosing suitable
coordinate on the torus fiber, we have
$\psi_k\circ...\circ\psi_0=\phi^\epsilon, \, \epsilon=\pm 1$
according to Theorem \ref{thm21}. Below we assume
$\psi_k\circ...\circ\psi_0=\phi$ (replace $\phi$ by $\phi^{-1}$ if
needed). Let $\widetilde{g}: M_\phi\setminus g^{-1}(T)\rightarrow
M_\phi\setminus T$ be the map induced by $g$. We have the following
commuted diagram:
\begin{equation}\label{dia1}
 \xymatrix{
  M_\phi\setminus g^{-1}(T) \ar[d]_{\bigcup\psi_i} \ar[r]^{\widetilde{g}}
                & M_\phi\setminus T \ar[d]^{\phi^\epsilon}  \\
  M_\phi  \ar[r]_{g}
                & M_\phi.             }
\end{equation}

Denote the restriction of $\widetilde{g}$ to $V_i$ by $g_i$. From
the commuted diagram in Figure 8, we have:
\begin{equation}\label{eq2}
g_{i+1}|_{T_{i+1}^+}\circ\psi_i=\phi^{\epsilon}\circ g_i|_{T_i^-},
\end{equation}
where $\epsilon=\pm 1$, $i=1,\cdots,k$ and if $i=k$ then $i+1$ is 1.
\begin{center}
\psfrag{g}[]{$g_i$} \psfrag{h}[]{$g_{i+1}$} \psfrag{i}[]{$\psi_i$}
\psfrag{j}[]{$\phi^{\epsilon}$} \psfrag{s}[]{$T_i^-$}
\psfrag{t}[]{$T_{i+1}^+$} \psfrag{u}[]{$V_i$}
\psfrag{v}[]{$V_{i+1}$}
\includegraphics[height=3in]{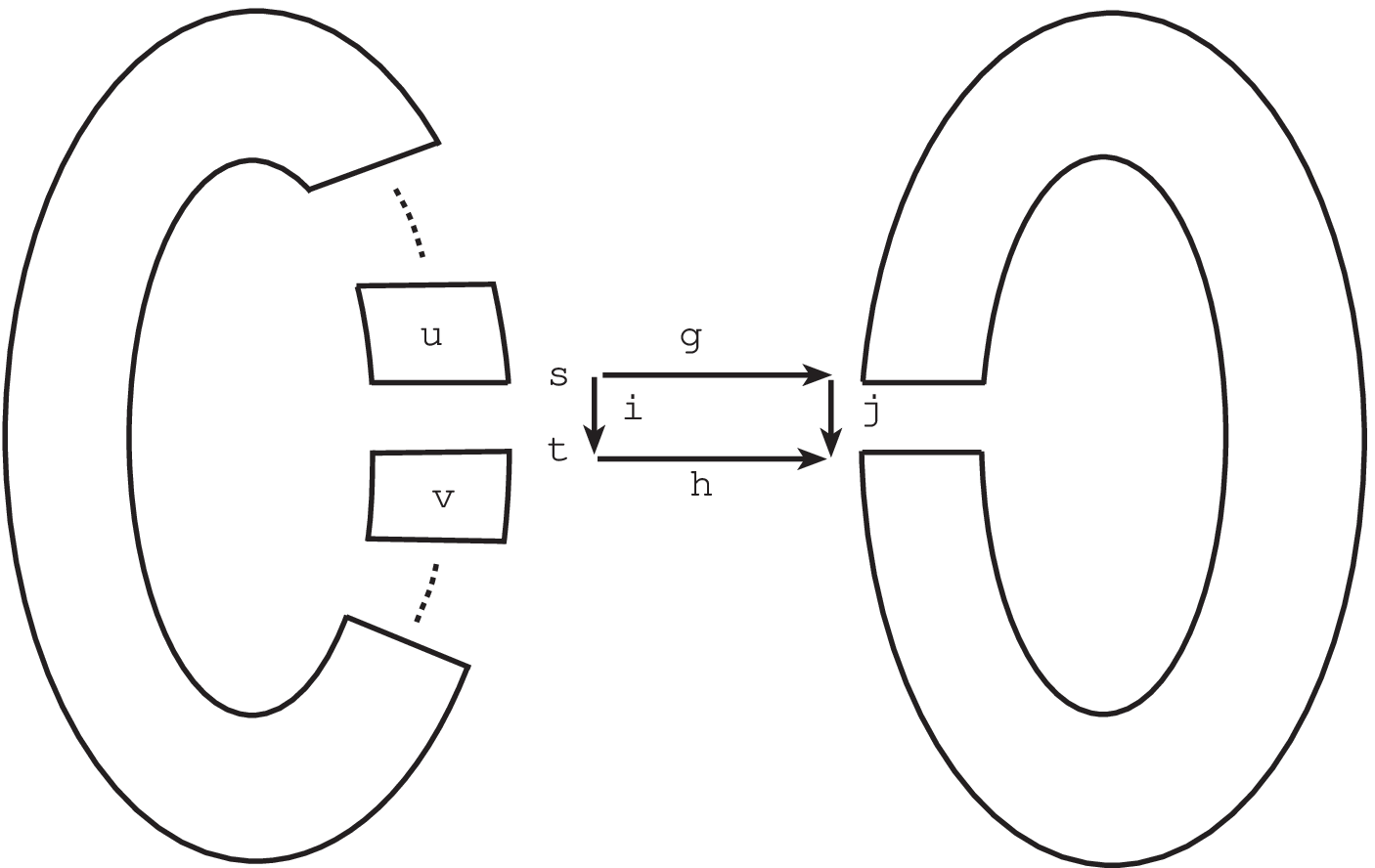}\\
\vskip 0.5 truecm \centerline{Figure 8: Non-zero degree self-map of
$M_\phi$}
\end{center}

Since $T_i^-$ is parallel to $T_i^+$, we can identify $\pi_1(T_i^-)$
with $\pi_1(T_i^+)$. Thus $(g_i|_{T_i^-})_*=(g_i|_{T_i^+})_*$ and
$(\psi_k)_*\cdots(\psi_1)_*=\phi_*$ on fundamental group. The
identity (\ref{eq2}) deduces that:
\begin{eqnarray*}
(g_1|_{T_1^+})_*\cdot\phi_*&=&(g_1|_{T_1^+})_*\cdot(\psi_k)_*\cdots(\psi_1)_*
\\&=&(g_{k+1}|_{T_{k+1}^+})_*\cdot(\psi_k)_*\cdots(\psi_1)_*
\\&=&\phi^{\epsilon}_*\cdot(g_k|_{T_k^-})_*\cdot(\psi_{k-1})_*\cdots(\psi_1)_*
\\&=&\phi^{\epsilon}_*\cdot(g_k|_{T_k^+})_*\cdot(\psi_{k-1})_*\cdots(\psi_1)_*
\\&=&\cdots
\\&=&(\phi^{\epsilon})_*^k\cdot(g_1|_{T_1^+})_*.
\end{eqnarray*}

Set $A=(g_1|_{T_1^+})_*$ and get:
\begin{equation}\label{eq3}
A\cdot\phi_*=(\phi^{\epsilon})_*^k\cdot A.
\end{equation}

Clearly $|deg(g)|=k| det(A)|$. The sign of $deg(g)$ is decided by
$\epsilon$ and the sign of $det(A)$. Thus
$l=deg(f)=deg(g)=k\cdot\epsilon\cdot det(A)$.

 Conversely, we set $\psi_1=\cdots=\psi_{k-1}=id,\ \psi_k=\phi$ and construct the map $\widetilde{g}: M_\phi\setminus g^{-1}(T)\rightarrow M_\phi\setminus T$
 such that $\widetilde{g}|_{V_i}=(\phi^{\epsilon\cdot (i-1)}\circ A)\times id: T\times I\rightarrow T\times I$
 for $i=1,\cdots,k$. This construction fits the commuted diagram (\ref{dia1}). Thus we get the quotient $g:M_\phi\rightarrow M_\phi$
  whose degree is equal to $k \cdot \epsilon \cdot det(A)$.
 \end{proof}

Suppose $A=\left(\begin{array}{cc}p & q
\\r & s\\\end{array}\right)$ where $p,q,r,s \in \mathbb{Z}$.
We use equation (\ref{eq3}) to solve $p,q,r,s$ and then can
determine $l$ by Theorem \ref{lm4}.

\textbf{(2.I)} If $\phi$ is Anosov which means the absolute value of
one eigenvalue of $\phi$ is larger than 1 while the other is less
than 1. In this case, the $k$ in the equation (\ref{eq3}) must be
equal to 1. We have:
$$
\left(\begin{array}{cc}p & q \\r & s\\\end{array}\right) \cdot
\left(\begin{array}{cc}a & b \\c & d\\\end{array}\right)=
\left(\begin{array}{cc}a & b \\c & d\\\end{array}\right)^{\epsilon}
\cdot \left(\begin{array}{cc}p & q \\r & s\\\end{array}\right).
$$

Solve this matric equation and get:

$$
A=\left\{ \begin{array}{l}
        \left(\begin{array}{cc}p & \frac{br}{c}
\\r & \frac{cp+(d-a)r}{c}\\\end{array}\right)\ (\epsilon=1) \\
        \left(\begin{array}{cc}p & \frac{p(d-a)-br}{c}
\\r & -p\\\end{array}\right)\ (\epsilon=-1)
\end{array} \right.
$$
where $\frac{br}{c}, \frac{(d-a)r}{c}, \frac{p(d-a)-br}{c}\in
\mathbb{Z}$.

By Theorem \ref{lm4}, we have:
$$
l=p^2+\frac{(d-a)}{c}\cdot pr-\frac{b}{c}\cdot r^2.
$$

\textbf{(2.II)} If $\phi$ is periodic, may assume  $\phi$ is either
$\left(\begin{array}{cc}-1 & -1
\\1 & 0\\\end{array}\right)$, or $\left(\begin{array}{cc}0 & -1
\\1 & 0\\\end{array}\right)$, or $\left(\begin{array}{cc}0 & -1
\\1 & 1\\\end{array}\right)$.

\textbf{(A)} If $\phi=$$\left(\begin{array}{cc}-1 & -1
\\1 & 0\\\end{array}\right)$ ($\phi$ has order 3), the equation (\ref{eq3}) means:
$$
A\cdot \phi_*=\left\{ \begin{array}{r}
        A\ (k\equiv0 \mod\ 3), \\
        \phi^{\epsilon}_* \cdot A\ (k\equiv1 \mod\ 3), \\
        \phi^{2\epsilon}_* \cdot A\ (k\equiv2 \mod\ 3).
        \end{array} \right.
$$

After solving all the above possible cases, we get:

$$
A= \left\{ \begin{array}{l}
        \left(\begin{array}{cc}p &
q\\-q & p-q\\\end{array}\right)\ (k\equiv1\mod3,\ \epsilon=1) \\
        \left(\begin{array}{cc}p &
q\\q-p & -p\\\end{array}\right)\ (k\equiv1\mod3,\ \epsilon=-1)\\
        \left(\begin{array}{cc}p &
q\\q-p & -p\\\end{array}\right)\ (k\equiv2\mod3,\ \epsilon=1) \\
        \left(\begin{array}{cc}p &
q\\-q & p-q\\\end{array}\right)\ (k\equiv2\mod3,\ \epsilon=-1)
\end{array} \right.
$$
If $k\equiv 0 \mod3$, we have $A=$$\left(\begin{array}{cc}0 & 0
\\0 & 0\\\end{array}\right)$, which induces degree 0 map.

By Theorem \ref{lm4}:

$$
l= \left\{ \begin{array}{r}
        k\cdot(p^2-pq+q^2)\ (k\equiv1\mod3), \\
        k\cdot(-p^2+pq-q^2)\ (k\equiv2\mod3).
\end{array} \right.
$$

It's easy to deduce that:
$$
l=(3t+1)(p^2-pq+q^2),\ t,p,q\in\mathbb{Z}.
$$

The same method is applied to the other two cases and we get:

\textbf{(B)} If $\phi=$$\left(\begin{array}{cc}0 & -1
\\1 & 0\\\end{array}\right)$, then:
$$
l=(4t+1)(p^2+q^2),\ t,p,q\in\mathbb{Z}.
$$

\textbf{(C)} If $\phi=$$\left(\begin{array}{cc}0 & -1
\\1 & 1\\\end{array}\right)$, then:
$$
l=(6t+1)(p^2-pq+q^2),\ t,p,q\in\mathbb{Z}.
$$

\section{The degrees of self maps of torus semi-bundles}

We are going to prove Theorem \ref{main-2} (ref. Proposition
\ref{coord-2}). We will assume that torus semi-bundle $N_\phi$
considered in this section has no torus bundle structure, otherwise
$D(N_\phi)$ is determined in Section 3.

 Suppose the degree of $f:N_\phi\rightarrow N_\phi$
is $l\neq 0$ and $T$ is a torus fiber of $N_\phi$. By \cite[Lemma
6.5]{He}, $f$ is homotopic to $g:N_\phi\rightarrow N_\phi$ such that
$g^{-1}(T)$ is incompressible in $N_\phi$. Thus by Lemma \ref{lm5}
and its proof (also ref. the proof of Theorem \ref{thm23}), we have
 $g^{-1}(T)$ is isotopic to either  a union of torus fibers,
 or a union of torus fibers of another semi-bundle structure which is isomorphic to the original one.
 Also the later case happen only if $N_\psi$ is a Nil manifold. Note by Theorem \ref{thm25} and the proof in Section 3 (1.II),
 Nil 3-manifolds admits no orientation reversing homeomorphism.

 Suppose now $g^{-1}(T)$ has $k$ connected
components, then $N_\phi\setminus g^{-1}(T)$ has two copies of $N$,
denoted by $V_0$ and $V_k$,  and $k-1$ copies of $T\times I$,
denoted by $V_i$, $i=1,\cdots,k-1$. Denote the boundaries of $V_0$
and $V_k$ by $T_0^-$ and $T_k^+$, the boundaries of  $V_i$ by
$T_i^+$ and $T_i^-$, $i=1,\cdots,k-1$,  and the gluing map from
$T_i^-$ to $T_{i+1}^+$ by $\psi_i$ ($i=0,\cdots,k-1$) see Figure 9.

\begin{center}
\psfrag{a}[]{$V_0$} \psfrag{b}[]{$V_1$} \psfrag{c}[]{$V_{k-1}$}
\psfrag{d}[]{$V_k$} \psfrag{e}[]{$N_1$} \psfrag{f}[]{$N_2$}
\psfrag{g}[]{$g_0$} \psfrag{h}[]{$g_1$} \psfrag{i}[]{$\psi_0$}
\psfrag{j}[]{$\phi^{\epsilon}$} \psfrag{k}[]{$T_0^-$}
\psfrag{l}[]{$T_1^+$} \psfrag{m}[]{$T_{k-1}^-$}
\psfrag{n}[]{$T_k^+$} \psfrag{o}[]{$\psi_{k-1}$}
\includegraphics[height=2.5in]{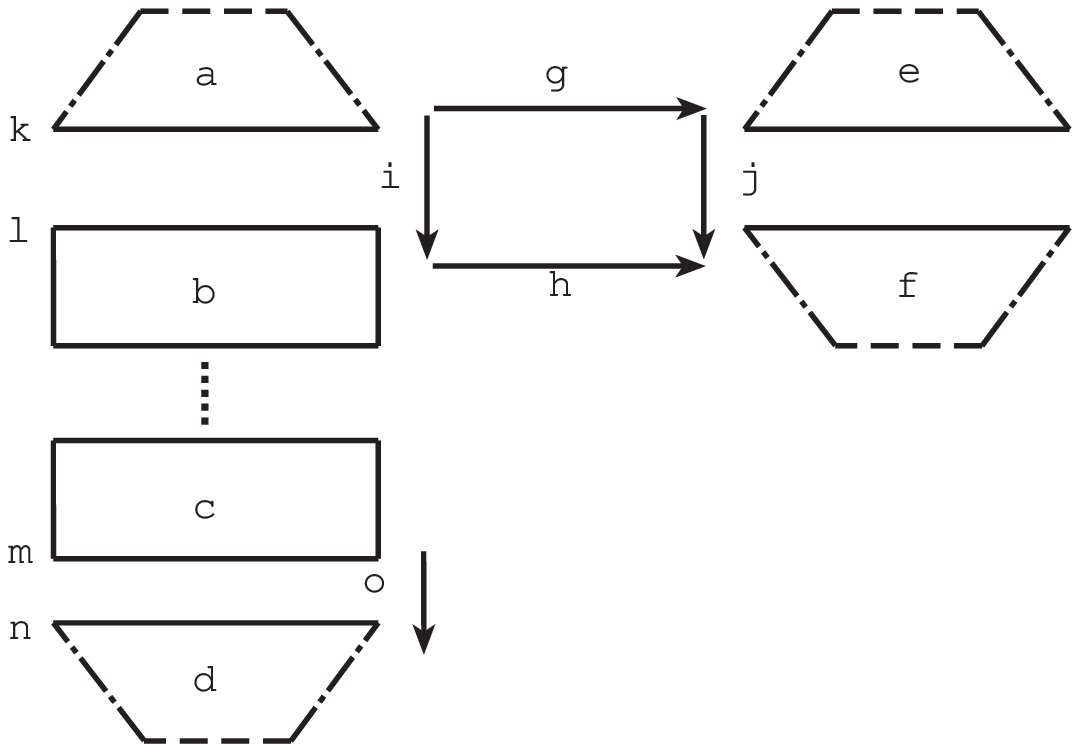}\\

\vskip 0.5 truecm \centerline{Figure 9: Non-zero degree self-map of
$N_\phi$}
\end{center}

Then $N_{\psi_{k-1}\circ...\circ\psi_0}=N_\phi$, and
$\psi_{k-1}\circ...\circ\psi_0=\phi^\epsilon, \, \epsilon=\pm 1$ by
Theorem \ref{thm23} (with a suitable orientation of the canonical
coordinate).  Below we assume $\psi_{k-1}\circ...\circ\psi_0=\phi$
(replace $\phi$ by $\phi^{-1}$ if needed). Let
$\widetilde{g}:N_\phi\setminus g^{-1}(T)\rightarrow N_\phi\setminus
T$ be the map induced by $g$, and we have commuted diagram:
\begin{equation}\label{dia2}
\xymatrix{ N_\phi\setminus g^{-1}(T) \ar[d]_{\bigcup\psi_i}
\ar[r]^{\widetilde{g}}
                & N_\phi\setminus T \ar[d]^{\phi^\epsilon}  \\
N_\phi  \ar[r]_{g}
                & N_\phi.             }
\end{equation}

Since $T_i^+$ is parallel to $T_i^-$, we can identity $\pi_1(T_i^+)$
with $\pi_1(T_i^-)$ ($i=0,\cdots,k-1$). Thus
$(\psi_{k-1})_*\cdots(\psi_0)_*=\phi_*$ on fundamental group. Denote
the restriction of $\widetilde{g}$ on $V_i$ by $g_i$. Then $g:
V_i\to N_1$ if $i$ even, and  $g: V_i\to N_2$ if $i$ odd.

\begin{lem}
Under the canonical basis $(l_0, l_\infty)$, $(g_0|_{T_0^-})_*$ is
of the form $\left(\begin{array}{cc}2m+1 & 0
\\0 & n \\ \end{array}\right)$ where $n\neq 0,m,n\in\mathbb{Z}$, and
so is $(g_k|_{T_k^+})_*$.
\end{lem}
\begin{proof} Let $g: N\to N$ be a proper map, we argue that under the basis $(l_0, l_\infty)$,
$(g|_{\partial N})_*$ is of the form $\left(\begin{array}{cc}2m+1 &
0
\\0 & n \\ \end{array}\right)$ where $n\neq 0,m,n\in\mathbb{Z}$.

Choose a presentation $\pi_1(N)=<a,b\ |\ a=bab>$ with $l_0=a^2$ and
$l_\infty=b$. Suppose  $g_{*}(a)=a^{m'} b^q,\ g_{*}(b)=a^p b^n$.
Since $g_{*}(a)=g_{*}(b)g_{*}(a)g_{*}(b)$, we get:
$$
a^{m'} b^q=a^p b^n a^{m'} b^q a^p b^n=a^{m'+2p}b^{(-1)^{m'+p}\cdot
n+(-1)^p\cdot q+n}.
$$

Thus:
$$
\left\{ \begin{array}{l}
         m'=m'+2p\\
         q=(-1)^{m'+p}\cdot n+(-1)^p\cdot q+n
        \end{array} \right.
        \Longrightarrow
\left\{ \begin{array}{l}
         p=0\\
         m'\ odd
\end{array} \right.
\ or \ \left\{ \begin{array}{l}
         p=0\\
         n=0.
\end{array} \right.
$$

Abandon the case that $p=n=0$ for $g_0$ is non-zero degree map and
let $m'=2m+1$, we get: $g_{*}(a)=a^{2m+1} b^q,\ g_{*}(b)=b^n$.

Since $\pi_1(\partial N)=<a^2,b\ |\ [a^2,b]=1>$ and
$g_{*}(a^2)=a^{2m+1} b^q a^{2m+1} b^q=a^{4m+2}$, we have
$$
(g|_{\partial N})_*=\left(
                   \begin{array}{cc}
                     2m+1 & 0 \\
                     0 & n \\
                   \end{array}
                 \right).
$$
\end{proof}

\begin{thm}
If $N_\phi$ has no torus bundle structure, then $N_\phi$ admits a
self map of degree $l\neq 0$ if and only if there exist a positive
integer $k$ and two integer matrices $A_1$, $A_2$ of form $\left(
                                                            \begin{array}{cc}
                                                              2m+1 & 0 \\
                                                              0 & n \\
                                                            \end{array}
                                                          \right),(m,n\in\mathbb{Z},n\neq 0)$
satisfying the following equation:
$$
A_2 \cdot \phi_*=\left\{
\begin{array}{r}
         (\phi^{-\epsilon}_* \cdot \tau_* \cdot \phi^{\epsilon}_*
\cdot \tau_*) ^{s-1} \cdot \phi^{-\epsilon}_*
\cdot \tau_* \cdot \phi^{\epsilon}_* \cdot A_1\ (k=2s),\\
(\phi^{\epsilon}_* \cdot \tau_* \cdot \phi^{-\epsilon}_* \cdot
\tau_*) ^{s} \cdot \phi^{\epsilon}_* \cdot A_1\ \ \ \ \ \ \ \ \ \
(k=2s+1),
                          \end{array} \right.
$$
such that $l=k\cdot \epsilon \cdot det(A_1)$ where $\epsilon=\pm 1$.
\end{thm}
\begin{proof}
From Figure 9, we know that:
\begin{equation}
g_{i+1}|_{T_{i+1}^+}\circ\psi_i=\left\{ \begin{array}{l}
         \ \phi^{\epsilon}\circ g_i|_{T_i^-}\ (i\equiv0\mod 2), \\
         \phi^{-\epsilon}\circ g_i|_{T_i^-}\ (i\equiv1\mod 2),
         \end{array}
\right.
\end{equation}
where $\epsilon=\pm 1$, $i=0,\cdots,k-1$.

Thus if $k=2s$ is even, then:
\begin{eqnarray}\label{eq32}
(g_k|_{T_k^+})_* \cdot \phi_*&=&(g_k|_{T_k^+})_* \cdot (\psi_{k-1})_* \cdots (\psi_0)_*\nonumber\hskip 1truecm \text{by Figure 9}\\
&=&\phi^{-\epsilon}_* \cdot (g_{k-1}|_{T_{k-1}^-})_* \cdot (\psi_{k-2})_* \cdots (\psi_0)_*\nonumber\hskip 1truecm \text{by (4.2)}\\
&=&\phi^{-\epsilon}_* \cdot \tau_* \cdot (g_{k-1}|_{T_{k-1}^+})_* \cdot (\psi_{k-2})_* \cdots (\psi_0)_*\nonumber\hskip 1truecm \text{by Lemma 2.8}\\
&=&\cdots\nonumber\\
&=&(\phi^{-\epsilon}_* \cdot \tau_* \cdot \phi^{\epsilon}_* \cdot
\tau_*) ^{s-1} \cdot \phi^{-\epsilon}_* \cdot \tau_* \cdot
\phi^{\epsilon}_* \cdot (g_0|_{T_0^-})_*.
\end{eqnarray}

If $k=2s+1$ is odd, then:
\begin{eqnarray}\label{eq33}
(g_k|_{T_k^+})_* \cdot \phi_*
&=&(g_k|_{T_k^+})_* \cdot (\psi_{k-1})_* \cdots (\psi_0)_*\nonumber\\
&=&\phi^{\epsilon}_* \cdot (g_{k-1}|_{T_{k-1}^-})_* \cdot (\psi_{k-2})_* \cdots (\psi_0)_*\nonumber\\
&=&\phi^{\epsilon}_* \cdot \tau_* \cdot (g_{k-1}|_{T_{k-1}^+})_* \cdot (\psi_{k-2})_* \cdots (\psi_0)_*\nonumber\\
&=&\cdots\nonumber\\
&=&(\phi^{\epsilon}_* \cdot \tau_* \cdot \phi^{-\epsilon}_* \cdot
\tau_*) ^{s} \cdot \phi^{\epsilon}_* \cdot (g_0|_{T_0^-})_*.
\end{eqnarray}

It is easy to see that
 $|deg(g)|=k| det(g_0|_{T_0^-})_*|$. The sign
of $deg(g)$ is decided by both $\epsilon$ and the sign of
$det(g_0|_{T_0^-})_*$. Thus $l=deg(f)=deg(g)=k\cdot\epsilon\cdot
det(g_0|_{T_0^-})_*$. Finally by applying Lemma 4.1, we finish the
proof of one direction of Theorem 4.2.

Conversely, if given $k,A_1,A_2$, then we can easily construct the
maps $g_0, g_k: N\rightarrow N$ such that $(g_0|_{T_0^-})_*=A_1,
(g_k|_{T_k^+})_*=A_2$. Set
$\psi_0=\cdot=\psi_{k-2}=id,\psi_{k-1}=\phi$ and $g_i: T\times
I\rightarrow N\ (i=1,\cdots,k-1)$ is a map such that:
 $$
 g_{i}|_{T_i^+}=\left\{
\begin{array}{l}
         \phi^{\epsilon}\circ g_{i-1}|_{T_{i-1}^-}\ \ \ (i\equiv1\mod 2), \\
         \phi^{-\epsilon}\circ g_{i-1}|_{T_{i-1}^-}\ (i\equiv0\mod
         2).
         \end{array}
\right.$$

Then $\widetilde{g}=\bigcup g_i$ fits the commutative diagram
(\ref{dia2}). Thus we get the quotient map $g: N_\phi\rightarrow
N_\phi$ of degree $k\cdot\epsilon\cdot det(A_1)$.
\end{proof}

 Given $\phi_*=\left(
                  \begin{array}{cc}
                    a & b \\
                    c & d \\
                  \end{array}
                \right)
\in GL_2(\mathbb{Z})$ and suppose $(g_0|_{T_0^-})_*=\left(
                  \begin{array}{cc}
                    2m+1 & 0 \\
                    0 & n \\
                  \end{array}
                \right)$,
$(g_k|_{T_k^+})_*=\left(
                  \begin{array}{cc}
                    2m^{'}+1 & 0 \\
                    0 & n^{'} \\
                  \end{array}
                \right)$ where $m,n,m^{'},n^{'}\in\mathbb{Z}$.
\\
\textbf{Case 1:} $abcd\neq 0, ad-bc=1$. (It should be noted that
$(\tau \phi \tau \phi^{-1})_*$ is Anosov.)

Since $g:N_\phi \to N_\phi$ satisfies $g^{-1}(T)$ is copies of torus
fiber, by Theorem \ref{thm25} $g$ can be lift to $g':M_{\tau \phi
\tau  \phi^{-1}} \to M_{\tau  \phi  \tau  \phi^{-1}}$. By the
argument of Anosov monodromy case in Section 3, the degree of $g'$
in the $S^1$ direction is 1. So we have $k=1$.

 By equation (\ref{eq33}), we have:
$$
(g_1|_{T_1^+})_* \cdot \phi_*=\phi^{\epsilon}_* \cdot
(g_0|_{T_0^-})_*.
$$

If $\epsilon=1$, then:

$$\left(
                  \begin{array}{cc}
                    2m'+1 & 0 \\
                    0 & n' \\
                  \end{array}
                \right)
                \cdot
\left(
                  \begin{array}{cc}
                    a & b \\
                    c & d \\
                  \end{array}
                \right)
                =
\left(
                  \begin{array}{cc}
                    a & b \\
                    c & d \\
                  \end{array}
                \right)
                \cdot
\left(
                  \begin{array}{cc}
                    2m+1 & 0 \\
                    0 & n \\
                  \end{array}
                \right).
$$

Solving this matrix equation we have:
$$
\left\{ \begin{array}{l}
         n=2m+1,\\
         m'=m,\\
         n'=2m+1.
         \end{array} \right.
$$

Thus $(g_0|_{T_0^-})_*=\left(
                  \begin{array}{cc}
                    2m+1 & 0 \\
                    0 & 2m+1 \\
                  \end{array}
                \right)$ which means:
$$
deg(g)=k\cdot\epsilon\cdot det((g_0|_{T_0^-})_*)=(2m+1)^2.
$$

If $\epsilon=-1$, then:

$$\left(
                  \begin{array}{cc}
                    2m'+1 & 0 \\
                    0 & n' \\
                  \end{array}
                \right)
                \cdot
\left(
                  \begin{array}{cc}
                    a & b \\
                    c & d \\
                  \end{array}
                \right)
                =
\left(
                  \begin{array}{cc}
                    a & b \\
                    c & d \\
                  \end{array}
                \right)^{-1}
                \cdot
\left(
                  \begin{array}{cc}
                    2m+1 & 0 \\
                    0 & n \\
                  \end{array}
                \right).
$$

Solving this matrix equation we have:
$$
\left\{ \begin{array}{l}
         n=-(2m'+1),\\
         (2m'+1)\cdot a=(2m+1)\cdot d,\\
         n'=-(2m+1).
         \end{array} \right.
$$

Suppose $(2m+1)=u\cdot \frac{a}{gcd(a,d)}$, then both $u$ and
$\frac{a}{gcd(a,d)}$ must be odd. Similarly, since $n=2m'+1=-u\cdot
\frac{d}{gcd(a,d)}$ is odd, then $\frac{d}{gcd(a,d)}$ is odd also.

Thus $(g_0|_{T_0^-})_*=\left(
                  \begin{array}{cc}
                    u\cdot \frac{a}{gcd(a,d)} & 0 \\
                    0 & -u\cdot \frac{d}{gcd(a,d)} \\
                  \end{array}
                \right)$ which means:
$$
deg(g)=k\cdot \epsilon\cdot det((g_0|_{T_0^-})_*)=u^2\cdot
\frac{ad}{gcd(a,d)^2}.
$$
This degree can be realized here if and only if
$\frac{ad}{gcd(a,d)^2}$ is odd.
\\

\textbf{Case 2:} $abcd=0$. Then there are three subcases.

\textbf{(2.I)} $\phi_*=\left(
                  \begin{array}{cc}
                    1 & 0 \\
                    z & 1 \\
                  \end{array}
                \right)
$.

In this case $N_\phi$ is a torus bundle which has been discussed in
section 3.

\textbf{(2.II)} $\phi_*=\left(
                  \begin{array}{cc}
                    0 & 1 \\
                    1 & z \\
                  \end{array}
                \right)
$, or equivalently $\left(
                  \begin{array}{cc}
                    z & 1 \\
                    1 & 0 \\
                  \end{array}
                \right)$.

When $z\neq 0$, we discuss the following four possible cases:

\textbf{(A)} If $\epsilon=1$ and $k=2s$ is even, then by equation
(4.3), we have the following equation:
$$\left(
                  \begin{array}{cc}
                    2m'+1 & 0 \\
                    0 & n' \\
                  \end{array}
                \right)
                \cdot
\left(
                  \begin{array}{cc}
                    0 & 1 \\
                    1 & z \\
                  \end{array}
                \right)
                =(-1)^s
\left(
                  \begin{array}{cc}
                    1 & zk \\
                    0 & -1 \\
                  \end{array}
                \right)
                \cdot
\left(
                  \begin{array}{cc}
                    2m+1 & 0 \\
                    0 & n \\
                  \end{array}
                \right).
$$
This equation has no solution.

\textbf{(B)} If $\epsilon=-1$ and $k=2s$ is even, then by equation
(4.3):
$$\left(
                  \begin{array}{cc}
                    2m'+1 & 0 \\
                    0 & n' \\
                  \end{array}
                \right)
                \cdot
\left(
                  \begin{array}{cc}
                    0 & 1 \\
                    1 & z \\
                  \end{array}
                \right)
                =(-1)^s
\left(
                  \begin{array}{cc}
                    1 & 0 \\
                    zk & -1 \\
                  \end{array}
                \right)
                \cdot
\left(
                  \begin{array}{cc}
                    2m+1 & 0 \\
                    0 & n \\
                  \end{array}
                \right).
$$
This equation has no solution either.

\textbf{(C)} If $\epsilon=1$ and $k=2s+1$ is odd, then by equation
(4.4):
$$\left(
                  \begin{array}{cc}
                    2m'+1 & 0 \\
                    0 & n' \\
                  \end{array}
                \right)
                \cdot
\left(
                  \begin{array}{cc}
                    0 & 1 \\
                    1 & z \\
                  \end{array}
                \right)
                =(-1)^s
\left(
                  \begin{array}{cc}
                    0 & 1 \\
                    1 & kz \\
                  \end{array}
                \right)
                \cdot
\left(
                  \begin{array}{cc}
                    2m+1 & 0 \\
                    0 & n \\
                  \end{array}
                \right).
$$

Solving this matrix equation:
$$
\left\{ \begin{array}{l}
         n=(-1)^s(2m'+1),\\
         n'=(-1)^s (2m+1),\\
         n'=(-1)^s kn.
         \end{array} \right.
$$
So $2m+1=kn$, thus $k$ is odd, if $k$ exists.

Then $(g_k|_{T_k^+})_*=\left(
                  \begin{array}{cc}
                    2m'+1 & 0 \\
                    0 & k(2m'+1) \\
                  \end{array}
                \right)$ which means:
$$
deg(g)=k\cdot \epsilon\cdot det((g_0|_{T_0^-})_*)=k\cdot
\epsilon\cdot det((g_k|_{T_k^+})_*)=k^2\cdot (2m'+1)^2.
$$

This degree is an odd square number. In another hand, when $k=1$,
all odd square number can be realized as a
degree:
$(g_k|_{T_k^+})_*=\left(
                  \begin{array}{cc}
                    2m'+1 & 0 \\
                    0 & 2m'+1 \\
                  \end{array}
                \right).$

\textbf{(D)} If $\epsilon=-1$ and $k=2s+1$ is odd, then by equation
(4.4):
$$\left(
                  \begin{array}{cc}
                    2m'+1 & 0 \\
                    0 & n' \\
                  \end{array}
                \right)
                \cdot
\left(
                  \begin{array}{cc}
                    0 & 1 \\
                    1 & z \\
                  \end{array}
                \right)
                =(-1)^s
\left(
                  \begin{array}{cc}
                    -zk & 1 \\
                    1 & 0 \\
                  \end{array}
                \right)
                \cdot
\left(
                  \begin{array}{cc}
                    2m+1 & 0 \\
                    0 & n \\
                  \end{array}
                \right).
$$
This equation has no solution.

When $z=0$, the same method will show that $deg(g)$ is odd, and all
odd numbers can be realized.

\textbf{(2.III)} $\phi_*=\left(
                  \begin{array}{cc}
                    1 & z \\
                    0 & 1 \\
                  \end{array}
                \right)
$.

In this case, $deg(g)$ can be determined as in case \textbf{(2.II)}.

{\ttfamily{A{\footnotesize CKNOWLEDGEMENTS.}}}  The paper is
enhanced by the referee's comments. The authors are supported by
grant No.10631060 of the National Natural Science Foundation of
China.

% ----------------------------------------------------------------
%\bibliographystyle{amsplain}
%\begin{thebibliography}{HKWZ}

\bibliographystyle{amsalpha}

Hongbin Sun

School of Mathematical Sciences, Peking University, Beijing 100871,
China

e-mail: hongbin.sun2331@gmail.com

\vskip 0.5 true cm

Shicheng Wang

School of Mathematical Sciences, Peking University, Beijing 100871,
China

e-mail: wangsc@math.pku.edu.cn

\vskip 0.5 true cm

Jianchun Wu

School of Mathematical Sciences, Peking University, Beijing 100871,
China

e-mail: wujianchun@math.pku.edu.cn
\end{document}